\documentclass[]{interact}

\usepackage{epstopdf}
\usepackage[caption=false]{subfig}
\usepackage{mathrsfs}
\usepackage{enumerate}
\usepackage{color}
\DeclareMathOperator*{\esssup}{ess\,sup}

\newcommand{\Atilda}{$\tilde{\mbox{A}}$}

\newtheorem{assumption}{A -}
\newtheorem{assumptiontilde}{$\tilde{\mbox{A}}$ -}
\allowdisplaybreaks

\usepackage[numbers,sort&compress]{natbib}
\bibpunct[, ]{[}{]}{,}{n}{,}{,}

\theoremstyle{plain}
\newtheorem{theorem}{Theorem}[section]
\newtheorem{lemma}[theorem]{Lemma}

\theoremstyle{definition}
\newtheorem{defn}[theorem]{Definition}
\newtheorem{example}[theorem]{Example}

\theoremstyle{remark}
\newtheorem{remark}{Remark}

\begin{document}


\title{Coercivity condition for higher moment a priori estimates for nonlinear SPDEs and existence of a solution under local monotonicity}

\author{
\name{Neelima\textsuperscript{a}\thanks{CONTACT Neelima. Email: N.Neelima@sms.ed.ac.uk} and David \v{S}i\v{s}ka\textsuperscript{b}}
\affil{\textsuperscript{a}School of Mathematics, University of Edinburgh, United Kingdom
and Ramjas College, University of Delhi, Delhi, India \textsuperscript{b}School of Mathematics, University of Edinburgh, United Kingdom}
}

\maketitle

\begin{abstract}
Higher moment a priori estimates for solutions to nonlinear SPDEs governed by 
locally-monotone operators are obtained under appropriate coercivity condition. 
These are then used to extend known existence and uniqueness results 
for nonlinear SPDEs under local monotonicity conditions 
to allow derivatives in the operator 
acting on the solution under the stochastic integral.
\end{abstract}


\section{Introduction}
Let $T>0$ be given, $(\Omega,\mathscr{F},(\mathscr{F}_t)_{t\in [0,T]},\mathbb{P})$ be a stochastic basis 
and $W:=(W_t)_{t\in [0,T]}$ be an infinite dimensional Wiener martingale
with respect to $(\mathscr{F}_t)_{t\in [0,T]}$, 
i.e. the coordinate processes $(W_t^j)_{t\in [0,T]}, j\in \mathbb{N}$ are 
independent $\mathscr{F}_t$-adapted Wiener processes and $W_t-W_s$ 
is independent of $\mathscr{F}_s$ for $s\leq~ t$. 
Further assume that $H$ is a separable Hilbert space,  
$V$ is a separable, reflexive Banach space 
embedded continuously and densely in $H$ and $V^*$ is the dual of $V$.
Identifying $H$ with $H^*$ using the Riesz representation and the inner product 
in $H$ one obtains the Gelfand triple 
\[
V \hookrightarrow H \equiv H^* \hookrightarrow V^*,
\]
where $\hookrightarrow$ denotes continuous and dense embeddings. 

Consider the stochastic evolution equation
\begin{equation}                                 \label{eq:see}
u_t=u_0+\int_0^tA_s(u_s)ds+\sum_{j=1}^\infty \int_0^tB_s^j(u_s)dW_s^j, \quad  t \in [0,T]\,,
\end{equation}
where the initial condition $u_0$ is an $H$-valued $\mathscr{F}_0$-measurable random variable.  
Moreover $A$ and $B^j, j \in \mathbb{N},$ are progressively measurable non-linear  operators  mapping $[0,T] \times \Omega \times V$ into $V^*$ and $H$ respectively. 
The exact assumptions will be given in Section~\ref{sec:assumptions}.
Further note that the formulation of~\eqref{eq:see} is equivalent
to considering the analogous equation driven by a cylindrical Wiener process,
see Appendix~\ref{sec:cyl}.

The nonlinear stochastic evolution equation~\eqref{eq:see} has been initially
studied in Pardoux~\cite{pardoux75} and Krylov and Rozovskii~\cite{krylov81},
where a priori estimates are proved, giving the second
moment estimates under what are now classical monotonicity, coercivity 
and growth assumptions.
This then allows the authors to obtain existence and uniqueness of solutions
to~\eqref{eq:see}.
One of the key results in~\cite{krylov81} is the theorem about It\^{o}'s formula for
the square of the norm of a continuous semimartingale in a Gelfand triple
obtained separately from the related stochastic evolution equation.
This theorem provides the continuity of the solution
in the pivot space of the Gelfand triple and is key to obtaining the a priori 
estimates and in proving the existence and uniqueness of the solution.
These, now classical results, have been generalized in a number of directions.
Of those one notes the inclusion of general c\`adl\`ag semimartingales as the
driving process in stochastic integral, 
see Gy\"ongy and Krylov~\cite{gyongy-krylov82} and Gy\"ongy~\cite{gyongy82}.
Closely related to the results in this paper is the work by
Liu and R\"{o}ckner~\cite{rockner10} (or \cite{rockner15}).
They extended the framework of Krylov and Rozovskii~\cite{krylov81} 
to stochastic evolution equations when the operators are only locally monotone 
and the operator $A$, which is the operator acting in the bounded variation term, 
satisfies a less restrictive growth condition.
To obtain a generalization in this direction
Liu and R\"{o}ckner~\cite{rockner10} 
need higher order moment estimates and to obtain them 
they place a restrictive assumption on the growth of the operator $B$ (i.e. \eqref{eq:growth_rock}), 
which is the operator acting on the solution under the stochastic integral.
As a consequence one may not have derivatives appearing in this operator.
The local monotonicity and coercivity conditions are further weakened 
in Liu and R\"ockner~\cite{liu13} but again at the expense of having a growth 
restriction on the operator $B$.
Moreover, Brze\'{z}niak, Liu and Zhu~\cite{brz14} extend the results in~\cite{rockner10} to include
equations driven by L\'evy noise but again with suboptimal growth restrictions 
on the operators appearing under the stochastic integrals (see also Remark~\ref{rem gamma_brez}).
Fully deterministic equations under local monotonicity assumptions are considered
in Liu~\cite{liu11}.

The main contribution of this paper is to identify appropriate coercivity 
assumption which allows one to obtain higher order moment estimates 
and to prove existence and uniqueness of solutions to~\eqref{eq:see} 
without the need to explicitly restrict the growth of the 
operator $B$. 
To be exact, we prove our results without requiring the first inequality in (1.2) in~\cite{rockner10} 
or equivalently in (5.2) in~\cite{rockner15} or (1.2) in~\cite{brz14}.

Examples of stochastic partial differential equations which do not fit the framework of Krylov and Rozovskii~\cite{krylov81} or Liu and R\"ockner~\cite{rockner10,rockner15} or Brze\'zniak, Liu and Zhu~\cite{brz14}
but which fit into the setting of this paper are given. 
See also Remark~\ref{rem:groB}. 
Finally, an example is considered that, together with results from Brze\'{z}niak and Veraar~\cite{brz12}, shows that the coercivity assumption 
identified in this paper is, in this context, the optimal one. 
See Example~\ref{ex sharp}.

This article is organized as follows. 
In Section~\ref{sec:assumptions} the main results about higher-order
moment estimates as well as existence and uniqueness of solutions are stated,
together with the assumptions required. 
Section~\ref{sec:apriori} is devoted to proving the a priori estimates and uniqueness of the solution. 
Galerkin discretization is used to obtain
a finite-dimensional approximation to~\eqref{eq:see} in Section~\ref{sec:moment}.
Moreover moment bounds for the
solutions of the finite-dimensional equations, uniform in the discretization 
parameter, are established.
These are used in Section~\ref{sec:mainresult} to prove existence of
solution to~\eqref{eq:see}.
Finally, Section \ref{sec:app} is devoted to examples of quasi-linear and semi-linear
stochastic partial differential equations which fit 
into the framework of this article.

\section{Assumptions and Main Results} \label{sec:assumptions}
Let $(\Omega,\mathscr{F},(\mathscr{F}_t)_{t\in [0,T]},\mathbb{P})$
be a filtered probability space satisfying the usual conditions, 
i.e., the probability space $(\Omega,\mathscr{F},\mathbb{P})$ is complete,
$\mathscr{F}_0$ contains all the $\mathbb{P}$-null sets that are in $\mathscr{F}$ and $(\mathscr{F}_t)_{t\in [0,T]}$ is right continuous. Let $W:=(W_t)_{t\in [0,T]}$ be an infinite dimensional Wiener martingale with respect to $(\mathscr{F}_t)_{t\in [0,T]}$.

Let $(X,|\cdot|_X)$ be a separable and reflexive Banach Space. 
For a given constant $p\in [1,\infty)$, $L^p(\Omega;X)$ denotes 
the Bochner--Lebesgue space of equivalence classes of random variables 
$x$ taking values in $X$ such that the norm
$$|x|_{L^p(\Omega;X)}:=(\mathbb{E}|x|_X^p)^\frac{1}{p}$$
is finite. Again, $L^p(0,T;X)$ denotes the Bochner--Lebesgue space of equivalence classes of $X$-valued measurable functions such that the norm
$$|x|_{L^p(0,T;X)}:= \Big(\int_0^T |x_t|_X^p\,dt \Big)^\frac{1}{p} $$
is finite
while $L^\infty(0,T;X)$ denotes the Bochner--Lebesgue space of 
$X$-valued measurable functions which are essentially bounded, i.e. 
$$
|x|_{L^\infty(0,T;X)}:=\esssup_{t\in (0,T)} |x_t|_X < \infty.
$$

Finally, $L^p((0,T)\times\Omega;X)$ denotes the Bochner--Lebesgue space of equivalence classes of $X$-valued stochastic processes which are progressively measurable and the norm
$$|x|_{L^p((0,T)\times\Omega;X)}:=\Big(\mathbb{E}\int_0^T |x_t|_X^p\,dt \Big)^\frac{1}{p}$$  
is finite.

Moreover, let $(H,(\cdot ,\cdot),|\cdot|_H)$ be a separable Hilbert space, identified with its dual
and let $(V,|\cdot|_V)$ denote a separable, reflexive Banach space 
embedded continuously and densely in $H$ with $(V^*,|\cdot|_{V^*})$ 
denoting its dual and $\langle\cdot,\cdot\rangle$ the duality pairing between $V$ and $V^*$.
Thus one has 
\[
V~ \hookrightarrow ~H~\equiv~ H^*~\hookrightarrow ~V^*
\] with continuous and dense embeddings.  

Let $A$ and $B^j$, $j \in \mathbb{N}$, be non-linear  operators mapping 
$[0,T] \times \Omega \times V$ into $V^*$ and $H$ respectively.
Assume that for all $v,w\in V$, the processes $(\langle A_t(v),w \rangle)_{t\in[0,T]}$ and $((B^j_t(v),w))_{t\in[0,T]}$ are progressively measurable. Since the concept of weak measurability and strong measurability of a mapping coincide if the codomain is separable, one gets that for all $v\in V,\,j\in \mathbb{N}$, $(A_t(v))_{t\in[0,T]}$ and $(B^j_t(v))_{t\in[0,T]}$ are progressively measurable.
Finally, $u_0$ is assumed to be a given $H$-valued $\mathscr{F}_0$-measurable random variable.

The following assumptions are made on the operators.
There exist constants 
$\alpha>1,\, \beta \geq 0,\,p_0 \geq \beta+2,\, \theta>0,\,K,\,L$ 
and a nonnegative $f\in L^\frac{p_0}{2}((0,T)\times \Omega; \mathbb{R})$ 
such that, almost surely, the following conditions hold 
for all $t\in [0,T]$.
\begin{assumption}[Hemicontinuity]\label{ass:hem} For all $y,x,\bar{x}$ in $V$, the map
\[
\varepsilon \mapsto \langle A_t(x+\varepsilon\bar{x}), y \rangle
\]
is continuous.
\end{assumption}
\begin{assumption}[Local Monotonicity]\label{ass:lmon} 
For all $x,\bar{x}$ in $V$,
\begin{align}
2\langle A_t(x)-A_t(\bar{x})&, x-\bar{x} \rangle + \sum_{j=1}^\infty |B^j_t(x)-B^j_t(\bar{x})|_H^2 \notag \\
& \leq  L(1+|\bar{x}|^\alpha_V)(1+|\bar{x}|^\beta_H)|x-\bar{x}|^2_H.\notag
\end{align}
\end{assumption}
\begin{assumption}[Coercivity]\label{ass:coer} 
For all $x$ in $V$,
$$2\langle A_t(x), x \rangle +(p_0-1)\sum_{j=1}^\infty |B^j_t(x)|_H^2 + \theta |x|_V^\alpha \leq f_t+K|x|^2_H.$$
\end{assumption}
\begin{assumption}[Growth of A]\label{ass:groA} For all $x$ in $V$, 
$$|A_t(x)|_{V^*}^{\frac{\alpha}{\alpha-1}}\leq (f_t+K|x|^\alpha_V)(1+|x|^\beta_H).$$
\end{assumption}
 
Note that, if $p_0=2$, i.e. $\beta=0$ and $L= 0$, then the conditions 
A-\ref{ass:hem} to A-\ref{ass:groA} reduce
to corresponding ones used in Krylov and Rozovskii~\cite{krylov81}.  

Throughout the article a generic constant $C$ will be used and it may 
change from line to line. 

\begin{remark}\label{rem:groB}
From Assumptions A-\ref{ass:coer} and A-\ref{ass:groA}, one obtains 
\begin{align}
&\sum_{j=1}^\infty |B^j_t(x)|_H^2 \leq C\big(1+f_t^\frac{p_0}{2}+|x|_H^{p_0} +|x|_V^\alpha +|x|_V^\alpha|x|_H^\beta\big) \notag 
\end{align}
almost surely for all $t\in[0,T]$ and $x\in V$.
Indeed, using H\"older's inequality, 
Young's inequality and Assumption A-\ref{ass:groA}, 
one obtains that almost surely for all $x\in V$ and $t \in [0,T]$,
\begin{align*}
|\langle A_t(x),x \rangle| & \leq \frac{\alpha-1}{\alpha} |A_t(x)|_{V^*}^{\frac{\alpha}{\alpha-1}}+\frac{1}{\alpha}|x|^\alpha_V \notag \\
& \leq \frac{\alpha-1}{\alpha} \big((f_t+K|x|^\alpha_V)(1+|x|^\beta_H)\big)+\frac{1}{\alpha}|x|^\alpha_V  \\
& \leq C\Big(f_t+|x|_V^\alpha+|x|_V^\alpha|x|_H^\beta+f_t^\frac{p_0}{2}+(1+|x|_H)^{p_0}\Big). 
\end{align*}
The above inequality along with Assumption A-\ref{ass:coer} gives the result.
\end{remark}

\begin{remark}
\label{rem:demicont}	
From Assumptions A-\ref{ass:hem}, A-\ref{ass:lmon} and A-\ref{ass:groA} one obtains that almost surely for all $t\in[0,T]$, the operator $A_t$ is demicontinuous, 
i.e. $v_n\to v \text{ in } V$ implies that 
$A_t(v_n)\rightharpoonup A_t(v)$ in $V^*$. 
This follows using similar arguments as in the proof of Lemma 2.1
in Krylov and Rozovskii~\cite{krylov81}.
\end{remark}
One consequence of this remark is that, progressive measurability 
of some process $(v_t)_{t\in [0,T]}$ implies 
the progressive measurability of the process $(A_t(v_t))_{t\in [0,T]}$.

\begin{defn}[Solution]\label{def:sol}
An adapted, continuous, $H$-valued process $u$ is called a solution of the 
stochastic evolution equation~\eqref{eq:see} if
\begin{enumerate}[i)]
\item $dt\times \mathbb{P}$ almost everywhere $u \in V$ and
\[
\mathbb{E}\int_0^T(|u_t|_V^\alpha+|u_t|_H^2)\, dt < \infty\,,
\]
\item for every $t\in [0,T]$ and $\phi \in V$
\[
(u_t,\phi) = (u_0,\phi) + \int_0^t \langle A_s(u_s),\phi \rangle ds + \sum_{j=1}^\infty \int_0^t (\phi,B_s^j(u_s))dW_s^j \quad a.s.
\] 
\end{enumerate}
\end{defn}
Note that the fact that $u$ is a continuous, $H$-valued process and i) in Definition~\ref{def:sol} implies that 
almost surely 
\[
\int_0^T \left(|u_t|_H^\beta + |u_t|_V^\alpha |u_t|_H^\beta \right)\, dt < \infty\,.
\]

The following are the main results of this article.

\begin{theorem}[A priori estimates]\label{thm:apriori}
If $u$ is a solution of~\eqref{eq:see} and Assumptions A-\ref{ass:coer} and A-\ref{ass:groA} hold, then
\begin{equation} \label{eq:as1}
\begin{split}
\sup_{t\in [0,T]}\mathbb{E}|u_t|_H^{p_0}
  +  \mathbb{E}\int_0^T|u_t|_H^{p_0-2}|u_t|_V^\alpha dt
  & \leq C\mathbb{E}\Big(|u_0|_H^{p_0}+ \int_0^T  f_s^\frac{p_0}{2}ds\Big) \quad \text{for } p_0>2,\\
  \sup_{t\in [0,T]}\mathbb{E}|u_t|_H^2+\mathbb{E}\int_0^T|u_t|_V^\alpha dt & \leq C\mathbb{E}\Big(|u_0|_H^2 +  \int_0^T f_s ds\Big).
\end{split}	
\end{equation}
Moreover,
\begin{equation}                                 
 \begin{split}\label{eq:as2}
& \mathbb{E}\sup_{t\in [0,T]}|u_t|_H^{2}\leq C\mathbb{E}\Big(|u_0|_H^2+ \int_0^T f_s ds\Big) \\
\text{and} \qquad & \mathbb{E}\sup_{t\in [0,T]}|u_t|_H^{p_0r}\leq C\mathbb{E}\Big(|u_0|_H^{p_0}+ \int_0^T f_s^\frac{p_0}{2}ds\Big)^r,
\end{split}
\end{equation}
for any $r\in (0,1)$, where $C$ depends only on $p_0,K,T,r$ and $\theta$.

\end{theorem}

Note that if $p_0 > 2$ then one cannot make use of the Burkholder--Davis--Gundy inequality to prove~\eqref{eq:as2}.
Indeed, in this case one would end up with a higher moment on the right-hand side than on the left when trying to prove the a priori estimate.
One avoids this problem by using Lenglart's inequality (see, e.g. Lemma~3.2 in Gy\"ongy and Krylov \cite{gyongy-krylov03}) but this means one can only get~\eqref{eq:as2} for $2\leq p < p_0$. 

\begin{theorem}[Uniqueness of solution] \label{thm:unique}
Let Assumptions A-\ref{ass:lmon} to A-\ref{ass:groA} hold
and $u_0\in L^{p_0}(\Omega; H)$. 
If $u$ and $\bar{u}$ are two solutions of \eqref{eq:see}, then the processes $u$ and $\bar{u}$ are indistinguishable, i.e. 
\begin{equation}
 \mathbb{P}\Big(\sup_{t\in[0,T]}|u_t-\bar{u}_t|_H=0\Big)=1.             \notag
\end{equation}
\end{theorem}

\begin{theorem}[Existence of solution] \label{thm:main}
If Assumptions A-\ref{ass:hem} to A-\ref{ass:groA} hold and $u_0\in L^{p_0}(\Omega; H)$, then the stochastic evolution equation~\eqref{eq:see} has a unique solution.
\end{theorem}

At first glance Assumption A-\ref{ass:coer} (equivalently \Atilda-\ref{ass:coer-cyl}, see the Appendix) seems to be more restrictive
than the one used in Liu and R\"ockner \cite{rockner10} and the reader 
may conclude that our results do not cover some SPDEs that can be treated by \cite{rockner10}. 
However this is not the case. 
Given the growth condition on operator $B$ that has been assumed in \cite[Theorem 1.1, inequality~(1.2)]{rockner10}, Assumption \Atilda-\ref{ass:coer-cyl} follows immediately from their coercivity condition. 
Indeed, below we recall the coercivity condition (H3) 
and growth condition (1.2) used by Liu and R\"ockner \cite{rockner10}:
for all $(t,\omega)\in [0,T]\times \Omega$ and $x \in V$,
\begin{equation}\label{eq:coer_rock}
2\langle A_t(x), x \rangle +|B_t(x)|_{L_2(U,H)}^2 + \theta |x|_V^\alpha \leq f_t+K|x|^2_H.
\end{equation}
and
\begin{equation}\label{eq:growth_rock}
|B_t(x)|_{L_2(U,H)}^2 \leq C( f_t+|x|^2_H).
\end{equation}  
Then multiplying \eqref{eq:growth_rock} by $(p_0-2)$ and adding the equation obtained to \eqref{eq:coer_rock}, one obtains
$$2\langle A_t(x), x \rangle +(p_0-1)|B_t(x)|_{L_2(U,H)}^2 + \theta |x|_V^\alpha \leq \tilde{f}_t+\tilde{K}|x|^2_H$$
where, $\tilde{f}_t=f_t+C(p_0-2)f_t$ with $\tilde{f}\in L^\frac{p_0}{2}((0,T)\times \Omega; \mathbb{R})$  and $\tilde{K}=K+C(p_0-2)$
which implies \Atilda-\ref{ass:coer-cyl} holds.
Examples~\ref{ex1}, \ref{ex_quasi}, \ref{ex2} and~\ref{ex2'} show that the converse does not hold.  
Moreover Example~\ref{ex sharp} shows that our assumption A-\ref{ass:coer}
(which is equivalent to \Atilda-\ref{ass:coer-cyl}) is sharp.

\section{A priori Estimates and Uniqueness of Solution} \label{sec:apriori}

\begin{proof}[Proof of Theorem \ref{thm:apriori}]
Let $u$ be a solution to equation \eqref{eq:see} in the sense of Definition \ref{def:sol}. Then, applying the It\^{o}'s formula for the square of the norm (see, e.g., Theorem 3.2 in \cite{krylov81} or Theorem 4.2.5 in \cite{rockner07}), one obtains
\begin{equation}
\label{eq:itosq} 
\begin{split}
|u_t|_H^2=|u_0|_H^2+\int_0^t\Big(2\langle A_s(u_s),u_s\rangle&+\sum_{j=1}^\infty|B_s^j(u_s)|_H^2\Big)ds  \\
&+ 2\sum_{j=1}^\infty\int_0^t(u_s, B_s^j(u_s))dW_s^j
\end{split}
\end{equation}
almost surely for all $t\in [0,T]$. Notice that this is a real-valued It\^{o} process. Thus, by It\^{o}'s formula, 
\begin{align}
&d|u_t|_H^{p_0}= \frac{p_0}{2}|u_t|_H^{p_0-2}\Big(2\langle A_t(u_t),u_t\rangle+\sum_{j=1}^\infty|B_t^j(u_t)|_H^2\Big)\, dt \notag \\
&+ p_0|u_t|_H^{p_0-2}\sum_{j=1}^\infty(u_t, B_t^j(u_t))dW_t^j +\frac{p_0(p_0-2)}{2}|u_t|_H^{p_0-4}\sum_{j=1}^\infty|(u_t, B_t^j(u_t))|^2\, dt \notag 
\end{align}
almost surely for all $t\in [0,T]$, which on using Cauchy-Schwarz inequality gives
\begin{align}
d|u_t|_H^{p_0}\leq \frac{p_0}{2}|u_t|_H^{p_0-2}&\Big(2\langle A_t(u_t),u_t\rangle+(p_0-1)\sum_{j=1}^\infty|B_t^j(u_t)|_H^2\Big)dt \notag \\
&+ p_0|u_t|_H^{p_0-2}\sum_{j=1}^\infty(u_t, B_t^j(u_t))dW_t^j. \label{eq:ito_p0} 
\end{align}
One aims to apply Lenglart's inequality (see, e.g., Lemma~3.2 and Remark 3.3 in \cite{gyongy-krylov03}). 
To that end 
let $\tau$ be some stopping time.
Moreover, to estimate the term containing the stochastic integral in~\eqref{eq:ito_p0}, one needs 
a sequence $(\sigma_n)_{n\in\mathbb{N}}$ of stopping times converging to 
$T$ as $n\to~\infty$, defined by
\begin{equation}
\sigma_n:=\inf\{t\in[0,T]:|u_t|_H>n\}\wedge T \label{eq:stoptime}.
\end{equation}
By using Assumption A-\ref{ass:coer} 
and Young's inequality in~\eqref{eq:ito_p0}, one obtains
\begin{align}
|u_{t\wedge\sigma_n\wedge \tau}|_H^{p_0}& \leq |u_0|_H^{p_0}+\frac{p_0}{2}\int_0^{t\wedge\sigma_n\wedge \tau}|u_s|_H^{p_0-2}\Big(f_s+K|u_s|_H^2-\theta|u_s|_V^\alpha\Big)ds \notag \\
&\quad+ p_0\sum_{j=1}^\infty \int_0^{t\wedge\sigma_n\wedge \tau}|u_s|_H^{p_0-2}(u_s,B_s^j(u_s))dW_s^j \notag \\
&\leq |u_0|_H^{p_0}+ \int_0^{t\wedge\sigma_n\wedge \tau} f_s^\frac{p_0}{2}ds +\frac{p_0-2}{2}\int_0^{t\wedge\sigma_n\wedge \tau}|u_s|_H^{p_0}ds\notag \\
&\quad +\frac{p_0}{2}K \int_0^{t\wedge\sigma_n\wedge \tau}|u_s|_H^{p_0}ds -\theta\frac{p_0}{2}\int_0^{t\wedge\sigma_n\wedge \tau}|u_s|_H^{p_0-2}|u_s
|_V^\alpha ds \notag \\
&\quad \quad+ p_0 \sum_{j=1}^\infty \int_0^{t\wedge\sigma_n\wedge \tau} |u_s|_H^{p_0-2}(u_s,B_s^j(u_s))dW_s^j. \notag 
\end{align}
Thus,
\begin{align}
|u_{t\wedge\sigma_n\wedge \tau}&|_H^{p_0}+\theta\frac{p_0}{2}\int_0^{t\wedge\sigma_n\wedge \tau}|u_s|_H^{p_0-2}|u_s
|_V^\alpha ds \notag \\
&\leq |u_0|_H^{p_0}+ \int_0^{t\wedge\sigma_n\wedge \tau}\!\! f_s^\frac{p_0}{2}ds+\frac{p_0(K\!+\!1)\!-2\!}{2}\int_0^{t}\!\!\mathbf{1}_{\{s\leq \sigma_n\wedge\tau\}}\!|u_s|_H^{p_0}ds\notag \\
&\quad \quad +p_0 \!\sum_{j=1}^\infty \int_0^{t\wedge\sigma_n }\!\! \mathbf{1}_{\{s\leq\tau\}}|u_s|_H^{p_0-2}(u_s,B_s^j(u_s))dW_s^j. \label{eq:beforeBGD} 
\end{align}
 Then in view of Remark \ref{rem:groB} and the fact that $u$ is a solution of equation \eqref{eq:see}, it follows that
$$\mathbb{E}\sum_{j=1}^\infty\int_0^{t\wedge\sigma_n}\mathbf{1}_{\{s\leq \tau\}}|u_s|_H^{p_0-2}(u_s,B_s^j(u_s))dW_s^j=0. $$
Therefore, taking expectation in \eqref{eq:beforeBGD}, one obtains
\begin{align}
\mathbb{E}&|u_{t\wedge\sigma_n \wedge \tau}|_H^{p_0}+\theta\frac{p_0}{2}\mathbb{E}\int_0^{t\wedge \sigma_n\wedge\tau}\!\!|u_s|_H^{p_0-2}|u_s|_V^\alpha ds \notag \\
&\leq \mathbb{E}|u_0|_H^{p_0}+ \mathbb{E}\int_0^{T}\!\! f_s^\frac{p_0}{2}ds+\frac{p_0(K+1)-2}{2}\mathbb{E}\int_0^{t}|u_{s\wedge\sigma_n\wedge\tau}|_H^{p_0}ds.  \label{eq:truegronwall}
\end{align}
From this Gronwall's lemma yields
\begin{align}
\label{eq:apriori-after-gronwall}
\mathbb{E}|u_{t\wedge\sigma_n\wedge\tau}|_H^{p_0}
\leq e^{\frac{p_0(K+1)-2}{2}T}\mathbb{E}\Big(|u_0|_H^{p_0}+ \int_0^T f_s^\frac{p_0}{2}ds\Big)
\end{align}
for all $t\in[0,T]$.
Letting $n\to \infty$ and using Fatou's lemma, one obtains
\begin{align}
                                                          \label{eq:beforeYor}
\mathbb{E}|u_{t\wedge\tau}|_H^{p_0}
&\leq e^{\frac{p_0(K+1)-2}{2}T}\mathbb{E}\Big(|u_0|_H^{p_0}+ \int_0^T f_s^\frac{p_0}{2}ds\Big)
\end{align}
for all $t\in[0,T]$. 
Using Lenglart's inequality, with the process 
$\left(|u_{T\wedge t}|_H^{p_0}\right)_{t\geq 0}$, one gets 
\[
\mathbb{E}\sup_{t\in[0,T]}|u_t|_H^{p_0r}\leq \frac{r}{1-r}e^{\frac{p_0(K+1)-2}{2}T}\mathbb{E}\Big(|u_0|_H^{p_0}+ \int_0^T f_s^\frac{p_0}{2}ds\Big)
\]
for any $r\in (0,1)$, which proves second inequality in \eqref{eq:as2}.

In order to prove~\eqref{eq:as1}, 
the estimate~\eqref{eq:apriori-after-gronwall} is used in the 
right-hand side of~\eqref{eq:truegronwall} with $\tau=T$ and 
with $n\to \infty$. 
One thus obtains 
\begin{equation*}
\mathbb{E}|u_t|_H^{p_0}
+\theta\frac{p_0}{2}\mathbb{E}\int_0^t|u_s|_H^{p_0-2}|u_s|_V^\alpha ds
\leq C\mathbb{E}\Big(|u_0|_H^{p_0}+ \int_0^T f_s^\frac{p_0}{2}ds\Big) \label{eq:apriori}
\end{equation*}
for all $t\in [0,T]$.
If Assumption A-\ref{ass:coer} holds for some $p_0\geq \beta+2$, then it holds for $p_0=2$ as well. 
Thus, using the stopping times $(\sigma_n)_{n\in \mathbb{N}}$
in~\eqref{eq:itosq} and taking expectation, one obtains, using the 
same localizing argument as before, that 
\begin{align*}
\mathbb{E}|u_t|_H^2+\theta\mathbb{E}\int_0^t |u_s|_V^\alpha ds \leq \mathbb{E}\Big(|u_0|_H^2+\int_0^Tf_sds\Big)+\mathbb{E}\int_0^tK|u_s|_H^2ds. \label{eq:vbound}
\end{align*}
for all $t\in[0,T]$. Application of Gronwall's lemma yields,
\[
\sup_{t\in[0,T]}\mathbb{E}|u_t|_H^2
\leq C\mathbb{E}\Big(|u_0|_H^2+ \int_0^T f_s ds\Big)
\]
which in turn gives
\[
\theta\mathbb{E}\int_0^T |u_s|_V^\alpha ds \leq C\mathbb{E}\Big(|u_0|_H^2+\int_0^T f_s \,ds\Big)\]
and hence~\eqref{eq:as1} holds.

To complete the proof it remains to show first inequality in \eqref{eq:as2}. 
This is done using the same argument as in 
Krylov and Rozovskii~\cite{krylov81}. 
It is included here for convenience of the reader. 
Considering the sequence of stopping times $\sigma_n$ 
defined in \eqref{eq:stoptime} and using Remark~\ref{rem:groB} 
along with Definition \ref{def:sol}, one observes that the 
stochastic integral in the right-hand side of \eqref{eq:itosq}
is a local martingale.
Thus invoking the Burkholder--Davis--Gundy inequality, one gets
\begin{align}
\mathbb{E}\sup_{t\in[0, T]}\Big|\sum_{j=1}^\infty \int_0^{t\wedge \sigma_n}(u_s, B_s^j(u_s))dW_s^j\Big|\leq 4\mathbb{E}\Big( \sum_{j=1}^\infty \int_0^{T\wedge \sigma_n}|(u_s, B_s^j(u_s))|^2 ds\Big)^\frac{1}{2}. \notag 
\end{align} 
Further, on using Cauchy--Schwarz inequality, Remark \ref{rem:groB} and Young's inequality one obtains
\begin{align}
\mathbb{E}\sup_{t\in[0, T]}&\Big|\sum_{j=1}^\infty \int_0^{t\wedge \sigma_n}(u_s, B_s^j(u_s))dW_s^j\Big|\leq 4\mathbb{E}\Big( \sum_{j=1}^\infty \int_0^{T\wedge \sigma_n}|u_s|_H^2 |B_s^j(u_s)|_H^2 ds\Big)^\frac{1}{2} \notag \\
& \leq 4\mathbb{E}\Big(\sup_{t\in[0, T]}|u_{t\wedge\sigma_n}|_H^2  \int_0^{T\wedge \sigma_n}\big(f_s+|u_s|_H^2+ |u_s|_V^\alpha \big) ds\Big)^\frac{1}{2} \notag \\
&\leq \epsilon \mathbb{E}\sup_{t\in[0, T]}|u_{t\wedge\sigma_n}|_H^2+C\mathbb{E}\int_0^{T\wedge \sigma_n}\big(f_s+|u_s|_H^2+ |u_s|_V^\alpha \big) ds. \label{eq:p2bdg}
\end{align}
Moreover, taking supremum and  then expectation in \eqref{eq:itosq} and using Assumption A-\ref{ass:coer} along with \eqref{eq:p2bdg}, one obtains  
\begin{align*}
\mathbb{E}\sup_{t\in[0, T]}|u_{t\wedge \sigma_n}|_H^2  \leq & \epsilon \mathbb{E}\sup_{t\in[0, T]}|u_{t\wedge\sigma_n}|_H^2 \\
&+C\Big(\mathbb{E}|u_0|_H^2+ \mathbb{E}\int_0^{T}\!\! f_s\,ds+\mathbb{E}\int_0^{T}|u_s|_V^\alpha ds+\sup_{t\in[0, T]}\mathbb{E}|u_{t}|_H^2\Big). 
\end{align*}
Finally, by choosing $\epsilon$ small and using \eqref{eq:as1} for $p_0=2$, one obtains
\begin{align}
\mathbb{E}\sup_{t\in[0, T]}|u_{t\wedge \sigma_n}|_H^2 \leq C\Big(\mathbb{E}|u_0|_H^2+ \mathbb{E}\int_0^{T}\!\! f_s\,ds\Big) \notag
\end{align}
which on allowing $n\to\infty$ and using Fatou's lemma finishes the proof.
\end{proof} 

\begin{defn}
\label{def psi}
Let $\Psi$ be defined as the collection of
$V$-valued and $\mathscr{F}_t$-adapted processes $\psi$ satisfying
\[
\int_0^T\!\!\!\rho(\psi_s)ds< \infty\,\,\, \text{ a.s.}\,, 
\]
where 
\begin{equation}
\label{eq rho}
\rho(x):=L(1+|x|_V^\alpha)(1+|x|_H^\beta)	
\end{equation}
for all $x\in V$.
\end{defn}

Note that if $u$ is a solution to~\eqref{eq:see} then $u\in \Psi$. 

\begin{remark}\label{rem:welldef}
For any $\psi \in \Psi$ and $v\in L^2(\Omega ,C([0,T]; H))$,
\begin{align}
\mathbb{E}\Big[\int_0^te^{-\int_0^s \rho(\psi_r)dr}&\rho(\psi_s)|v_s|_H^2ds\Big] \leq \mathbb{E}\sup_{s\in[0, t]}|v_s|_H^2\int_0^te^{-\int_0^s \rho(\psi_r)dr} \rho(\psi_s)ds \notag \\
&=\mathbb{E}\sup_{s\in[0, t]}|v_s|_H^2[1-e^{-\int_0^t \rho(\psi_r)dr}] \leq \mathbb{E}\sup_{s\in[0, t]}|v_s|_H^2 <\infty. \notag
\end{align}
\end{remark}
This remark justifies the existence of the bounded variation integrals 
appearing in the proof of uniqueness that follows.

\begin{proof}[Proof of Theorem \ref{thm:unique}] Consider two solutions $u$ and $\bar{u}$ of \eqref{eq:see}. Thus,
\begin{align}
u_t - \bar u_t =\int_0^t \left(A_s(u_s)-A_s(\bar{u}_s)\right)\,ds
+\sum_{j=1}^\infty \int_0^t \left(B_s^j(u_s)-B_s^j(\bar{u}_s)\right)\,dW_s^j  
\end{align}
almost surely for all $ t \in [0,T]$.
Using the It\^{o}'s formula and the product rule one obtains 
\begin{align*}
d&\Big(e^{-\int_0^t \rho(\bar{u}_s)\,ds}|u_t-\bar{u}_t|_H^2 \Big)= e^{-\int_0^t \rho(\bar{u}_s)ds}\big[d|u_t-\bar{u}_t|_H^2-\rho(\bar{u}_t)|u_t-\bar{u}_t|_H^2\,dt\big]\\
&= e^{-\int_0^t \rho(\bar{u}_s)ds}\bigg[\Big(2\langle A_t(u_t) - A_t(\bar{u}_t),u_t\!-\!\bar{u}_t\rangle+\sum_{j=1}^\infty|B_t^j(u_t) - B_t^j(\bar{u}_t)|_H^2\Big)\,dt  \\
& \quad+ \sum_{j=1}^\infty 2\big(u_t-\bar{u}_t,B_t^j(u_t)\!-\!B_t^j(\bar{u}_t)\big) dW_t^j -\rho(\bar{u}_t)|u_t-\bar{u}_t|_H^2dt\bigg]     
\end{align*}
almost surely for all $ t \in [0,T]$. With Assumption A-\ref{ass:lmon} one sees that, with $t_n := t\wedge \sigma_n$ and $\sigma_n:=\inf\{t\in[0,T]:|u_t|_H>n\}\wedge \inf\{t\in[0,T]:|\bar{u}_t|_H>n\}\wedge T$,
\begin{align*}
e^{-\int_0^{t_n} \rho(\bar{u}_s)ds}&|u_{t_n}-\bar{u}_{t_n}|_H^2 \\
&\leq 2\sum_{j=1}^\infty \int_0^{t_n} e^{-\int_0^s \rho(\bar{u}_r)dr}\big( u_s-\bar{u}_s,B_s^j(u_s)-B_s^j(\bar{u}_s)\big) dW_s^j. 
\end{align*}
Then,
\begin{equation*}
\mathbb{E}[e^{-\int_0^{t_n} \rho(\bar{u}_s)ds}|u_{t_n}-\bar{u}_{t_n}|_H^2]\leq 0.
\end{equation*} 
Letting $n\to \infty$ and using Fatou's lemma one concludes that
for all $t \in [0,T]$ one has $\mathbb{P}(|u_t - \bar u_t|_H^2=0)=1$.
This, together with the continuity of $u-\bar u$ in $H$, concludes the 
proof.
\end{proof}   

\section{A priori Estimates for Galerkin Discretization} \label{sec:moment}

Existence of solution to stochastic evolution equation \eqref{eq:see} 
will now be shown using the Galerkin method.
Consider a Galerkin scheme $(V_m)_{m\in \mathbb{N}}$ for $V$, i.e. for each $m\in \mathbb{N}$, $V_m$ is an $m$-dimensional subspace of $V$ such that $V_m \subset V_{m+1} \subset V$ and $\cup_{m\in \mathbb{N}} V_m$ is dense in $V$. Let $\{\phi_i : i=1,2,\ldots m\}$ be a basis of $V_m$.
Assume that for each $m\in\mathbb{N}$, $u_0^m$ is a $V_m$-valued 
$\mathscr{F}_0$-measurable random variable satisfying 
\begin{equation}
\label{eq:initial_data}
\sup_{m\in\mathbb{N}}\mathbb{E}|u_0^m|_H^{p_0}<\infty \,\,\, \mbox{and} \,\,\, \mathbb{E}|u_0^m-u_0|_H^2\to 0 \,\,\,\text{as}\,\,\, m\to\infty . 
\end{equation}
It is always possible to obtain such an approximating sequence. For example, consider $\{\phi_i\}_{i\in\mathbb{N}}\!\subset \! V$ forming an orthonormal basis in $H$ and for each $m\in \mathbb{N}$, take $u_0^m=~\Pi_mu_0$ where $\Pi_m:H\to V_m$ are the projection operators.

For each $m\in \mathbb{N}$ and $\phi_i \in V_m$, $i=1,2,\ldots,m$, consider the stochastic differential equation:

\begin{equation}
(u_t^m,\phi_i)=(u_0^m,\phi_i)+\int_0^t \langle A_s(u_s^m),\phi_i \rangle ds+ \sum_{j=1}^m \int_0^t(\phi_i,B_s^j(u_s^m))dW_s^j   \label{eq:sde}
\end{equation}
 almost surely for all $t\in[0,T]$.
Using the results on solvability of stochastic differential equations 
in finite dimensional space (see, e.g., Theorem 3.1 in \cite{krylov81}), together with 
Assumptions A-\ref{ass:hem} to A-\ref{ass:groA} and
Remark~\ref{rem:demicont}, there exists a 
unique adapted and continuous (and thus progressively measurable) $V_m$-valued process $u^m$ satisfying \eqref{eq:sde}.

\begin{lemma}[A priori Estimates for Galerkin Discretization] \label{lem:galbound}
Suppose that~\eqref{eq:initial_data} and  
Assumptions A-\ref{ass:coer} and A-\ref{ass:groA} hold.
Then there is $C$ independent of $m$, such that
\begin{equation}
\label{eq:as1-G}	
\sup_{t\in [0,T]}\mathbb{E}|u_t^m|_H^{p_0}+\mathbb{E}\int_0^T\!\!|u_t^m|_V^\alpha \, dt+\mathbb{E}\int_0^T\!\!|u_t^m|_H^{p_0-2} |u_t^m|_V^\alpha \, dt \leq C,
\end{equation}
\begin{equation}
\label{eq:as2-G}	
\mathbb{E}\sup_{t\in [0,T]}|u_t^m|_H^{p}\leq C,	
\end{equation}
with $p=2$ in case $p_0=2$ (i.e. $\beta=0$) and $p\in[2,p_0)$ if $p_0>2$,
\begin{align}
&\mathbb{E}\int_0^{T} |A_s(u^m_s)|_{V^*}^\frac{\alpha}{\alpha-1} ds \leq C, \label{eq:galboundA}\\
&\mathbb{E}\sum_{j=1}^\infty \int_0^{T} |B^j_s(u^m_s)|_H^2 ds \leq C. \label{eq:galboundB}
\end{align}
\end{lemma}
\begin{proof}
Proof of~\eqref{eq:as1-G} and~\eqref{eq:as2-G} is almost a repetition of the proof of analogous results in Theorem \ref{thm:apriori}. Indeed, for each $m\in\mathbb{N}$, one can define a sequence of stopping times 
\[
\sigma_n^m:=\inf\{t\in[0,T]:|u_t^m|_H>n\}\wedge T
\] 
and repeat the steps of Theorem \ref{thm:apriori} by replacing $u_t$ with $u_t^m$ and $\sigma_n$ with $\sigma_n^m$. 
There are two main points to be noted. 
The first is that the stochastic integral appearing on right-hand side of \eqref{eq:itosq}, 
with $u_t$ replaced by $u_t^m$, is a local martingale for each $m\in\mathbb{N}$. 
Indeed, on a finite dimensional space, all norms are equivalent and hence
\[
\mathbb{E}\int_0^{T\wedge\sigma_n^m} |u_t^m|_V^\alpha dt 
\leq C_m \mathbb{E}\int_0^{T\wedge\sigma_n^m}  n^\alpha dt<\infty
\]
with some constant $C_m$.
The second point is that, since
\[
\sup_{m\in \mathbb{N}}\mathbb{E}|u_0^m|^{p_0} <~\infty,
\]
one can take a constant independent of $m$ to obtain~\eqref{eq:as1-G} 
and~\eqref{eq:as2-G}.

The estimates \eqref{eq:galboundA} and \eqref{eq:galboundB} can be proved as below. One obtains from Assumption A-\ref{ass:groA}, that
\begin{equation*}
\begin{split}
I & :=\mathbb{E}\int_0^{T} |A_s(u^m_s)|_{V^*}^\frac{\alpha}{\alpha-1} ds
\leq \mathbb{E}\int_0^{T}(f_s+K|u^m_s|^\alpha_V)(1+|u^m_s|^\beta_H)ds\\
& = \mathbb{E}\int_0^{T}f_s\, ds+ \mathbb{E}\int_0^{T}f_s|u^m_s|^\beta_Hds +K\mathbb{E}\int_0^{T}|u^m_s|^\alpha_V ds  \\
&\quad + K\mathbb{E}\int_0^{T}|u^m_s|^\alpha_V|u^m_s|^\beta_H \, ds\, . 
\end{split}
\end{equation*}
Using Young's inequality one can see that
\[
f_s+f_s|u^m_s|_H^\beta \leq \frac{4}{p_0}f_s^\frac{p_0}{2} 
+ \frac{p_0-2}{p_0} + \frac{p_0-2}{p_0}|u^m_s|_H^{\beta \frac{p_0}{p_0-2}}.
\]
Moreover, $|u^m_s|_H^\beta \leq (1+|u^m_s|_H)^{p_0-2}$, since $p_0 \geq \beta +2$.
Hence 
\begin{equation*}
\begin{split}
I & \leq \frac{4}{p_0}\mathbb{E}\int_0^{T}f_s^\frac{p_0}{2}ds+\frac{p_0-2}{p_0}T+ \frac{p_0-2}{p_0}\mathbb{E}\int_0^T|u^m_s|^{\beta \frac{p_0}{p_0-2}}_H ds\notag \\
&\quad +K\mathbb{E}\int_0^{T}|u^m_s|^\alpha_Vds+ K\mathbb{E}\int_0^T|u^m_s|^\alpha_V(1+|u^m_s|_H)^{p_0-2}ds.	
\end{split}	
\end{equation*}
Furthermore, applying H\"{o}lder's inequality,  
\begin{equation}
\label{eq:boundA}
\begin{split}
I&\leq   \frac{4}{p_0}\mathbb{E}\int_0^{T}f_s^\frac{p_0}{2}ds+\frac{p_0-2}{p_0}T+\frac{p_0-2}{p_0}T^\frac{p_0-2-\beta}{p_0-2}\Big(\mathbb{E}\int_0^T|u^m_s|^{p_0}_H ds\Big)^\frac{\beta}{p_0-2}  \\ 
&\quad +(2^{p_0-3}+1)K\mathbb{E}\int_0^{T}|u^m_s|^\alpha_Vds + 2^{p_0-3}K\int_0^T|u^m_s|^\alpha_V|u_s^m|_H^{p_0-2}ds \\
&\leq \frac{4}{p_0}\mathbb{E}\int_0^{T}f_s^\frac{p_0}{2}ds+\frac{p_0-2}{p_0}T+\frac{p_0-2}{p_0}T \sup_{0\leq s \leq T}\mathbb{E}|u^m_s|^{p_0}_H \\ 
&\quad +(2^{p_0-3}+1)K\mathbb{E}\int_0^{T}|u^m_s|^\alpha_Vds + 2^{p_0-3}K\int_0^T|u^m_s|^\alpha_V|u_s^m|_H^{p_0-2}ds\,, 
\end{split}	
\end{equation}
where one has used the fact $p_0\geq \beta+2$. By using~\eqref{eq:as1-G} in \eqref{eq:boundA}, one obtains \eqref{eq:galboundA}. 
Furthermore, by Remark \ref{rem:groB}, one gets
\begin{align}
\mathbb{E}\int_0^{T}\sum_{j=1}^\infty |B_s^j(u^m_s)|_H^2& ds \leq C\Big[T+\mathbb{E}\int_0^Tf_t^\frac{p_0}{2}ds +\mathbb{E}\int_0^T|u^m_s|_H^{p_0}ds\notag \\
&  +\mathbb{E}\int_0^T|u^m_s|_V^\alpha ds +\mathbb{E}\int_0^T|u^m_s|_V^\alpha(1+|u^m_s|_H)^{p_0-2} ds \Big] \notag \\
&  \quad\leq C\Big[T+\mathbb{E}\int_0^Tf_t^\frac{p_0}{2}ds+T\sup_{s\in[0, T]}\mathbb{E}|u^m_s|_H^{p_0} \notag \\
& \quad  \quad  +\mathbb{E}\int_0^T|u^m_s|_V^\alpha ds +\mathbb{E}\int_0^T|u^m_s|_V^\alpha|u^m_s|_H^{p_0-2} ds \Big] \notag
\end{align}
and hence by using~\eqref{eq:as1-G}, one gets \eqref{eq:galboundB}.
\end{proof}

\section{Existence of Solution} \label{sec:mainresult}

Having obtained the necessary a priori estimates, weakly convergent subsequences are extracted using the compactness argument. After that the local monotonicity condition is used to establish the existence of a solution to \eqref{eq:see}.

\begin{lemma} \label{lem:weaklimit}
Let Assumptions A-\ref{ass:lmon}, A-\ref{ass:coer}, A-\ref{ass:groA} 
and~\eqref{eq:initial_data} hold. 
Then there is a subsequence $(m_k)_{k\in \mathbb{N}}$ and
\begin{enumerate}[i)]
\item there exists a progressively measurable process
$u\in L^\alpha((0,T)\times \Omega ; V)$
such that 
\[u^{m_k}\rightharpoonup u\,\,\, \text{in}\,\,\,L^\alpha((0,T)\times \Omega ; V),
\]
\item there exists a progressively measurable process $a \in L^\frac{\alpha}{\alpha-1}((0,T)\times \Omega ; V^*)$ such that \[
A(u^{m_k})\rightharpoonup a\,\,\, \text{in}\,\,\, L^\frac{\alpha}{\alpha-1}((0,T)\times \Omega ; V^*),
\]
\item there exists a progressively measurable process $b\in L^2((0,T)\times \Omega ; \ell^2(H))$ such that 
\[
B (u^{m_k})\rightharpoonup b \,\,\, \text{in}\,\,\, L^2((0,T)\times \Omega ; \ell^2(H)).
\]

\end{enumerate}

\end{lemma}

\begin{proof}
The Banach spaces $L^\alpha((0,T)\times \Omega ; V),\quad L^\frac{\alpha}{\alpha-1}((0,T)\times \Omega ; V^*)$ and $L^2((0,T)\times \Omega ; \ell^2(H))$ are reflexive. Thus, due to Lemma \ref{lem:galbound}, there exists a subsequence $m_k$ (see, e.g., Theorem 3.18 in \cite{brezis10}) such that
\begin{itemize}
\item[(i)] $u^{m_k}\rightharpoonup v$ in $L^\alpha((0,T)\times \Omega ; V)$
\item[(ii)] $A(u^{m_k})\rightharpoonup a$ in $L^\frac{\alpha}{\alpha-1}((0,T)\times \Omega ; V^*)$
\item[(iii)] $(B^j(u^{m_k}))_{j=1}^{m_k}\rightharpoonup (b^j)_{j=1}^\infty$ in $L^2((0,T)\times \Omega ; \ell^2(H))$.
\end{itemize}
\end{proof}

Whilst not needed to prove results in this article, 
it is also possible to show that there is a subsequence of $(m_k)$, 
again denoted by $m_k$ such that $u^{m_k}$ converges weakly star to $u$
in $L^p(\Omega; L^\infty(0,T;H))$. 
This is a consequence of Lemma~\ref{lem:galbound} and Lemma~\ref{lem:weak_star}.

\begin{lemma} \label{lem:limiteq}
Let Assumptions A-\ref{ass:lmon}, A-\ref{ass:coer} and A-\ref{ass:groA} 
together with~\eqref{eq:initial_data} hold.
Then, 	
\begin{enumerate}[i)]
\item for $dt\times \mathbb{P}$ almost everywhere,
\begin{equation}
u_t=u_0+\int_0^t a_sds + \sum_{j=1}^\infty \int_0^t b^j_sdW_s^j \notag
\end{equation}
and moreover almost surely $u\in C([0,T];H)$ and for all $t$ 
\begin{equation}
|u_t|_H^2 = |u_0|_H^2 + \int_0^t \Big[ 2\langle a_s, u_s \rangle 
+ \sum_{j=1}^\infty |b^j_s|_H^2 \Big]\, ds 
+ 2\sum_{j=1}^\infty \int_0^t (u_s, b^j_s) dW^j_s \,. \label{eq:itoweak}
\end{equation}
\item Finally, $u\in L^2(\Omega;C([0,T];H))$.

\end{enumerate}
\end{lemma}

\begin{proof}
Using It\^{o}'s isometry, it can be shown that the stochastic integral is a bounded linear operator from $L^2((0,T)\times \Omega ; \ell^2(H))$ to $L^2((0,T)\times \Omega ; H)$ and hence maps a weakly convergent sequence to a weakly convergent sequence.
Thus, one obtains
\begin{equation}
 \sum_{j=1}^{m_k} \int_0^\cdot B^j_s(u^{m_k}_s)dW_s^j \rightharpoonup \sum_{j=1}^\infty \int_0^\cdot  b^j_s dW_s^j      \notag
\end{equation}
in $L^2((0,T)\times \Omega ; H)$, i.e. for any $\psi \in L^2((0,T)\times \Omega ; H)$,
\begin{equation}
 \mathbb{E}\int_0^T\!\!\Big(\sum_{j=1}^{m_k} \int_0^t B^j_s(u^{m_k}_s)dW_s^j,\psi(t)\Big)dt\rightarrow  \mathbb{E}\int_0^T\!\Big(\sum_{j=1}^\infty \int_0^t \!\!  b^j_s dW_s^j,\psi(t)\Big)dt.   \label{eq:stochastic}
\end{equation} 
Similarly, using Holder's inequality it can be shown that the Bochner integral is a bounded linear operator from $L^{\frac{\alpha}{\alpha-1}}((0,T)\times\Omega;V^*)$
 to $L^{\frac{\alpha}{\alpha-1}}((0,T)\times~\Omega;V^*)$ and is thus continuous with respect to weak topologies. Therefore, for any $\psi\in L^\alpha((0,T)\times\Omega;V)$,
 \begin{equation}
 \mathbb{E}\int_0^T\Big\langle \int_0^t A_s(u^{m_k}_s)ds,\psi(t)\Big\rangle dt\rightarrow  \mathbb{E}\int_0^T\Big\langle \int_0^t  a_s ds,\psi(t)\Big\rangle dt.     \label{eq:bochner}
\end{equation}
Fix $n\in \mathbb{N}$. Then for any $\phi \in V_n$ and an adapted real valued process $\eta_t$ bounded by a constant $C$, one has, for any $k\geq n$,  
\begin{align}
&\mathbb{E}\int_0^T \eta_t(u^{m_k}_t\!,\phi)\, dt \notag \\
&=\mathbb{E} \int_0^T \eta_t\Big((u^{m_k}_0,\phi)+\int_0^t \langle A_s(u^{m_k}_s),\phi \rangle\, ds + \sum_{j=1}^{m_k} \int_0^t (\phi,B_s^j(u^{m_k}_s))dW_s^j\Big)\,dt. \notag
\end{align} 
Taking the limit $k\rightarrow \infty$ and using~\eqref{eq:initial_data},  \eqref{eq:stochastic} and \eqref{eq:bochner}, one obtains
\begin{align}
&\mathbb{E}\int_0^T \eta_t(v_t,\phi)\,dt \notag \\
&=\mathbb{E}\int_0^T \eta_t\Big((u_0,\phi)+\int_0^t \langle a_s,\phi \rangle \, ds + \sum_{j=1}^\infty \int_0^t (\phi,b^j_s)\,dW_s^j \Big)\,dt\notag
\end{align} 
with any $\phi \in V_n$ and any adapted real valued process $\eta_t$ bounded by a constant $C$.
Since $\cup_{n\in \mathbb{N}}V_n$ is dense in $V$, one obtains 
\begin{equation}
v_t=u_0+\int_0^t a_sds + \sum_{j=1}^\infty \int_0^t b^j_s dW_s^j \label{eq:hsol}
\end{equation}
$dt\times \mathbb{P}$ almost everywhere. 
Now, using Theorem 3.2 on It\^o's formula from \cite{krylov81}, 
there exists an $H$-valued continuous modification $u$ of $v$ 
which is equal to the right hand side of \eqref{eq:hsol} almost surely for all $t\in[0,T]$.
Moreover~\eqref{eq:itoweak} holds almost surely for all $t\in [0,T]$. 
This completes the proof of part (i) of the lemma. 
It remains to prove part (ii) of  the lemma. 
To that end, consider the sequence of stopping times $\sigma_n$ defined for each $n\in\mathbb{N}$ by
$$
\sigma_n:=\inf\{t\in[0,T]:|u_t|_H>n\}\wedge T\,.
$$
From the Burkholder--Davis--Gundy inequality, one obtains
\begin{align}
\mathbb{E}\sup_{t\in[0, T]}\Big|\sum_{j=1}^\infty \int_0^{t\wedge \sigma_n}(u_s, b_s^j)dW_s^j\Big|\leq 4\mathbb{E}\Big( \sum_{j=1}^\infty \int_0^{T\wedge \sigma_n}|(u_s, b_s^j)|_H^2 ds\Big)^\frac{1}{2}\,. \notag 
\end{align}
Using Cauchy--Schwarz's and Young's inequalities leads to
\begin{align}
\mathbb{E}\sup_{t\in[0, T]}&\Big|\sum_{j=1}^\infty \int_0^{t\wedge \sigma_n}(u_s, b_s^j)dW_s^j\Big|\leq 4\mathbb{E}\Big( \sum_{j=1}^\infty \int_0^{T\wedge \sigma_n}|u_s|_H^2 |b_s^j|_H^2 ds\Big)^\frac{1}{2} \notag \\
& \leq 4\mathbb{E}\Big(\sup_{t\in[0, T]}|u_{t\wedge\sigma_n}|_H^2 \sum_{j=1}^\infty \int_0^{T\wedge \sigma_n}|b_s^j|_H^2 ds\Big)^\frac{1}{2} \notag \\
&\leq \epsilon \mathbb{E}\sup_{t\in[0, T]}|u_{t\wedge\sigma_n}|_H^2+C\mathbb{E}\sum_{j=1}^\infty\int_0^{T\wedge \sigma_n}|b_s^j|_H^2 ds. \label{eq:weakbdg}
\end{align}
Taking supremum and  then expectation in~\eqref{eq:itoweak} and using  H\"{o}lder's inequality along with \eqref{eq:weakbdg}, one obtains  
\begin{align}
\mathbb{E}\sup_{t\in[0, T]}|u_{t\wedge \sigma_n}|_H^2 \leq \mathbb{E}|u_0|_H^2 &+2 \Big(\mathbb{E}\int_0^{T}|a_s|^\frac{\alpha}{\alpha-1}ds\Big)^\frac{\alpha-1}{\alpha}\Big(\mathbb{E}\int_0^{T}|u_s|_V^\alpha ds\Big)^\frac{1}{\alpha}\notag \\
& +\epsilon \mathbb{E}\sup_{t\in[0, T]}|u_{t\wedge\sigma_n}|_H^2 +C \mathbb{E}\sum_{j=1}^\infty \int_0^T|b_s^j|_H^2ds\notag 
\end{align} 
which on choosing $\epsilon$ small enough gives
\begin{align}
&\mathbb{E}\sup_{t\in[0, T]}|u_{t\wedge \sigma_n}|_H^2 \notag \\
& \leq C \Big[\mathbb{E}|u_0|_H^2 +   \Big(\mathbb{E}\int_0^{T}|a_s|^\frac{\alpha}{\alpha-1}ds\Big)^\frac{\alpha-1}{\alpha}\Big(\mathbb{E}\int_0^{T}|u_s|_V^\alpha ds\Big)^\frac{1}{\alpha} +\mathbb{E}\sum_{j=1}^\infty \int_0^T|b_s^j|_H^2ds \Big]. \notag
\end{align} 
Finally taking $n \to \infty$ and using Fatou's lemma, one obtains
$$\mathbb{E}\sup_{t\in[0, T]}|u_t|_H^2<\infty\,.$$
This concludes the proof.
\end{proof}

From now onwards, the processes $v$ and $u$ will be denoted by $u$ for notational convenience. 
In order to prove that the process $u$ is the solution of equation \eqref{eq:see}, it remains to show that $dt\times \mathbb{P}$ almost everywhere $A(v)=a$ and $B^j(v)=b^j$ for all $j\in \mathbb{N}$. 
Recall that $\Psi$ and $\rho$ were given in Definition~\ref{def psi}.

\begin{proof}[Proof of Theorem \ref{thm:main}] 
For $\psi \in L^\alpha((0,T)\times \Omega ; V)\cap \Psi \cap L^2(\Omega ; C([0,T]; H))$, using the product rule and It\^{o}'s formula one obtains
\begin{equation}
\label{eq:ito1}
\begin{split}
&\mathbb{E}\big(e^{-\int_0^t \rho(\psi_s)ds}|u_t|_H^2\big) -\mathbb{E}(|u_0|_H^2)   \\
& = \mathbb{E}\Big[\int_0^te^{-\int_0^s \rho(\psi_r)dr}\Big(2\langle a_s,u_s\rangle  
+\sum_{j=1}^\infty |b^j_s|_H^2-\rho(\psi_s)|u_s|_H^2\Big)ds\Big] 
\end{split}
\end{equation}
and
\begin{equation}
\label{eq:ito}
\begin{split}
\mathbb{E}\big(e^{-\int_0^t \rho(\psi_s)ds}|u_t^{m_k}|_H^2\big)&-\mathbb{E}(|u_0^{m_k}|_H^2) \\
&=\mathbb{E}\Big[\int_0^te^{-\int_0^s \rho(\psi_r)dr}\Big(2\langle A_s(u_s^{m_k}),u_s^{m_k}\rangle \\
& \quad +\sum_{j=1}^{m_k} |B_s^j(u_s^{m_k})|_H^2-\rho(\psi_s)|u_s^{m_k}|_H^2\Big)ds\Big] 
\end{split}
\end{equation}
for all $t\in [0,T]$.
Note that in view of Remark \ref{rem:welldef}, all the integrals are well defined in what follows.
Moreover,
\begin{align}
&\mathbb{E}\Big[\int_0^te^{-\int_0^s \rho(\psi_r)dr}\Big(2\langle A_s(u_s^{m_k}),u_s^{m_k}\rangle +\sum_{j=1}^{m_k} |B_s^j(u_s^{m_k})|_H^2-\rho(\psi_s)|u_s^{m_k}|_H^2\Big)ds\Big] \notag \\
&=\mathbb{E}\Big[\int_0^te^{-\int_0^s \rho(\psi_r)dr}\Big(2\langle A_s(u_s^{m_k})-A_s(\psi_s),u_s^{m_k}-\psi_s\rangle +2\langle A_s(\psi_s),u_s^{m_k} \rangle \notag \\
& +2 \langle A_s(u_s^{m_k})-A_s(\psi_s),\psi_s\rangle +\sum_{j=1}^{m_k} \big|B_s^j(u_s^{m_k})-B_s^j(\psi_s)\big|_H^2-\sum_{j=1}^{m_k}|B_s^j(\psi_s)|_H^2 \notag \\
&  +2\sum_{j=1}^{m_k}\big(B_s^j(u_s^{m_k}),B_s^j(\psi_s)\big)-\rho(\psi_s)\left[|u_s^{m_k}-\psi_s|_H^2-|\psi_s|_H^2 
+2(u_s^{m_k},\psi_s)\right]\!\Big)ds\Big] \,. \notag
\end{align}
Now one can apply the local monotonicity Assumption A-\ref{ass:lmon} to see that
\begin{align}
& \mathbb{E}\big(e^{-\int_0^t \rho(\psi_s)ds}|u_t^{m_k}|_H^2\big)-\mathbb{E}(|u_0^{m_k}|_H^2) \notag \\
&\leq \mathbb{E}\Big[\int_0^te^{-\int_0^s \rho(\psi_r)dr}\Big( 2\langle A_s(\psi_s),u_s^{m_k}\rangle+2 \langle A_s(u_s^{m_k})-A_s(\psi_s),\psi_s\rangle \notag \\
&  -\sum_{j=1}^{m_k}|B_s^j(\psi_s)|_H^2 +2\sum_{j=1}^{m_k}\big(B_s^j(u_s^{m_k}),B_s^j(\psi_s)\big)+\rho(\psi_s)\big[|\psi_s|_H^2 
 -2(u_s^{m_k},\psi_s)\big] \Big)ds\Big]. \notag
\end{align}
Integrating over $t$ from $0$ to $T$, letting $k \rightarrow \infty$ and using the weak lower semicontinuity of the norm one obtains,
\begin{align}
&\mathbb{E}\Big[\int_0^T \big(e^{-\int_0^t \rho(\psi_s)ds}|u_t|_H^2-|u_0|_H^2\big)dt\Big] \notag \\
& \leq \liminf_{k\rightarrow\infty}\mathbb{E}\Big[\int_0^T\!\big(e^{-\int_0^t \rho(\psi_s)ds}|u_t^{m_k}|_H^2-|u_0^{m_k}|_H^2\big)dt\Big] \notag \\
&\leq \mathbb{E}\Big[\int_0^T\int_0^te^{-\int_0^s \rho(\psi_r)dr}\Big( 2\langle A_s(\psi_s),u_s\rangle+2 \langle a_s-A_s(\psi_s),\psi_s\rangle -\sum_{j=1}^\infty|B_s^j(\psi_s)|_H^2 \notag \\
& \quad \quad  +2\sum_{j=1}^\infty(b^j_s,B_s^j(\psi_s))+\rho(\psi_s)\left[|\psi_s|_H^2 
 -2(u_s,\psi_s)\right] \Big)dsdt\Big].  \label{eq:weaklimits}
\end{align}
Integrating from $0$ to $T$ in~\eqref{eq:ito1} and combining this 
with~\eqref{eq:weaklimits} leads to
\begin{align}
\mathbb{E}\Big[\int_0^T&\int_0^t e^{-\int_0^s \rho(\psi_r)dr}\Big( 2\langle a_s-A_s(\psi_s),u_s-\psi_s \rangle \notag \\
&\quad+\sum_{j=1}^\infty|B_s^j(\psi_s)-b^j_s|_H^2 -\rho(\psi_s)|u_s-\psi_s|_H^2  \Big)dsdt\Big]\leq 0. \label{eq:limitB}
\end{align}
Further, using the Definition~\ref{def psi} and Lemma~\ref{lem:weaklimit},
\[
u\in L^\alpha((0,T)\times \Omega ; V)\cap \Psi \cap L^2(\Omega ; C([0,T]; H))\,.
\]
Taking $\psi = u$ in \eqref{eq:limitB}, one obtains that 
$B(u)=b$ in $L^2((0,T)\times \Omega;\ell^2(H))$.
Let $\eta\in L^\infty((0,T)\times \Omega;\mathbb{R})$, $\phi\in V$, 
$\epsilon\in (0,1)$ and let $\psi=u-\epsilon \eta \phi$. Then
from~\eqref{eq:limitB} one obtains that
\begin{align}
\mathbb{E}\Big[\int_0^T \!\! \int_0^t & e^{-\int_0^s \rho(u_r-\epsilon \eta_r \phi)dr}\Big( 
2\epsilon \langle a_s-A_s(u_s-\epsilon \eta_s \phi),\eta_s \phi \rangle \notag \\
&\quad-\epsilon^2 \rho(u_s-\epsilon \eta_s \phi)|\eta_s \phi|_H^2  \Big)dsdt\Big]\leq 0. \label{eq:limitA}
\end{align}
Dividing by $\epsilon$, letting $\epsilon\rightarrow  0$, 
using Lebesgue dominated convergence theorem and 
Assumption A-\ref{ass:hem} leads to
\begin{align}
&\mathbb{E}\Big[\int_0^T \int_0^t e^{-\int_0^s \rho(u_r)dr}
2 \eta_s \langle a_s-A_s(u_s),\phi \rangle dsdt\Big]\leq 0. \notag
\end{align}
Since this holds for any $\eta\in L^\infty((0,T)\times \Omega ;\mathbb{R})$ and $\phi\in V$, one gets that $A(u)= a$ in $L^\frac{\alpha}{\alpha-1}((0,T)\times \Omega;V^*)$ which concludes the proof.
\end{proof}
 
\section{Examples} \label{sec:app}
In this section, some examples of stochastic evolution equations are presented
which fit in the framework of this article and yet do not satisfy the assumptions of \cite{krylov81, rockner10}.

Throughout the section, $\mathbb{R}^d$ denotes a $d$-dimensional Euclidean space.
For $x,y \in \mathbb R^d$, the inner product is denoted by $xy$.  
Let $\mathscr{D}\subseteq \mathbb{R}^d$ be an open bounded domain with smooth
boundary. 
Then for any $p\geq 1, \,\, L^p(\mathscr{D})$ is the Lebesgue space
of equivalence classes of real valued measurable functions $u$ defined
on $\mathscr{D}$ such that the norm 
\[|u|_{L^p}:=\Big(\int_\mathscr{D}|u(x)|^p dx\Big)^\frac{1}{p}
\] 
is finite.
For $i\in\{1,2,\ldots,d\}$, let $D_i$ denote the distributional derivative along the $i$-th coordinate in $\mathbb{R}^d$.
Further, let $\nabla:=(D_1,D_2,\ldots,D_d)$ denote the gradient.
Finally, $W^{1,p}(\mathscr{D})$ is the Sobolev space of real valued functions $u$,
defined on $\mathscr{D}$, such that the norm
\[
|u|_{1,p}:=\Big(\int_\mathscr{D} \big(|u(x)|^p + |\nabla u(x)|^p\big)\, dx\Big)^\frac{1}{p}
\]
is finite.

Let $C_0^\infty(\mathscr{D})$ be the space of smooth functions with compact support in $\mathscr{D}$. 
Then, the closure of $C_0^\infty(\mathscr{D})$ in $W^{1,p}(\mathscr{D})$ with respect to the norm $|\cdot|_{1,p}$ is denoted by  
$W_0^{1,p}(\mathscr{D})$. 
Friederichs' inequality (see, e.g. Theorem 1.32 in \cite{roubicek05}) implies
that the norm
\[
|u|_{W_0^{1,p}}:=\Big(\int_\mathscr{D} |\nabla u(x)|^p\,dx\Big)^\frac{1}{p} 
\]
is equivalent to $|u|_{1,p}$
and this  equivalent norm $|u|_{W_0^{1,p}}$ will be used throughout this section. 
Moreover, let $W^{-1,p}(\mathscr{D})$ denotes the dual of $W^{1,p}_0(\mathscr{D})$
and let $|\cdot|_{W^{-1,p}}$ be the norm on this dual space.
It is well known that 
\[ 
W_0^{1,p}(\mathscr{D}) \hookrightarrow L^2(\mathscr{D})\equiv (L^2(\mathscr{D}))^\ast \hookrightarrow W^{-1,p}(\mathscr{D}),
\]
where $\hookrightarrow$ denotes continuous and dense embeddings, 
is a Gelfand triple.
Finally, define $\Delta:W^{1,2}_0(\mathscr{D}) \to W^{-1,2}(\mathscr{D})$ by 
\[
\langle \Delta u, v \rangle := -\int_\mathscr{D} \nabla u(x) \nabla v(x) dx
\,\,\,\, \forall v\in W^{1,2}_0(\mathscr{D})\,. 
\]
Clearly 
\begin{equation}
|\Delta u|_{W^{-1,2}}\leq |u|_{W_0^{1,2}}  \label{eq:delta}
\end{equation}
and so the operator is linear and bounded.

The following consequence of Gagliardo--Nirenberg inequality (see, e.g., Theorem 1.24 in \cite{roubicek05}) will be needed in the examples presented below. 
If $d=2$, then there exists a constant $C$ such that
\begin{equation}
|u|_{L^4}\leq C|u|_{L^2}^\frac{1}{2}|u|_{W_0^{1,2}}^\frac{1}{2}.    \label{eq:d2}
\end{equation}
Further, if $d=1$, then there exists a constant $C$ such that
\begin{equation*}
|u|_{L^4}\leq C|u|_{L^2}^\frac{3}{4}|u|_{W_0^{1,2}}^\frac{1}{4} 
\leq C|u|_{L^2}^\frac{1}{2}|u|_{W_0^{1,2}}^\frac{1}{2}\,. 
\end{equation*}

\begin{example}[Semi-linear equation] 
\label{ex1}
Let $\gamma$ be a constant such that $\gamma^2 < \frac{1}{3}$.
For $i=1,2,\ldots,d$, let $g_i:\mathbb{R\to \mathbb{R}}$\, 
be bounded and Lipschitz continuous and $h_i:\mathbb{R}\to \mathbb{R}$ 
be Lipschitz continuous.
Let $f:\mathbb{R}\to\mathbb{R}$ be a continuous function such that
\[
|f(x)|\leq K(1+|x|^3) \,\,\text{ and }\,\, 
(f(x)-f(y))(x-y) \leq K(1+|y|^2)|x-y|^2\,\,\, \forall x,y\in \mathbb{R}.
\] 
Consider the stochastic partial differential equation
\begin{equation}                                                 \label{ex:1}
du_t= \big(\Delta u_t+  g(u_t)\, \nabla u_t + f(u_t) \big)\,dt 
+ \big(\gamma \nabla u_t+h(u_t)\big)\,dW_t
\,\,\, \text{on}\,\,\, (0,T)\times \mathscr{D},      
\end{equation}
where $u_t = 0$ on $\partial \mathscr{D}$, 
$u_0$ is a given $\mathcal{F}_0$-measurable random variable
and $\Delta$ is the Laplace operator.
Moreover $W$ is an $\mathbb R^d$-valued Wiener process.
It will now be shown that such an equation, in its weak form, fits the assumptions
of the present article.

Let $A:W^{1,2}_0(\mathscr{D}) \to W^{-1,2}(\mathscr{D})$ and $B^i:W^{1,2}_0(\mathscr{D}) \to L^2(\mathscr{D})$ be given by
\[
A(u):=\Delta u+ g(u)\nabla u + f(u) \text{ and } B^i(u):=\gamma D_iu+h_i(u) 
\]
for $i=1,2,\ldots d$.
The next step is to show that these operators satisfy
the Assumptions A-\ref{ass:hem} to A-\ref{ass:groA}.
One immediately notices that A-\ref{ass:hem} holds, 
in particular, since $g$ and $f$ are continuous. 

One now wishes to verify the local monotonicity condition.
By using the assumptions imposed on $f$ and $g$ one can see 
for $u,v\in W_0^{1,2}(\mathscr{D})$, upon application of H\"older's inequality, 
that
\begin{equation*}
\begin{split}
&\langle A(u)-A(v),u-v \rangle \\
&=-|u-v|_{W_0^{1,2}}^2 + \langle g(u)\nabla(u - v),u-v\rangle + \langle (\nabla v)(g(u)-g(v)),u - v\rangle\\ 
& \quad \quad +\langle f(u)-f(v), u-v\rangle\\
&\leq -|u-v|_{W_0^{1,2}}^2+C|u-v|_{W_0^{1,2}}|u-v|_{L^2} 
+ C|v|_{W_0^{1,2}}|u-v|_{L^4}^2 \\
& \quad \quad + C|u-v|_{L^2}^2 + C |v|_{L^{4}}^2 |u-v|_{L^4}^2 \, . 
\end{split}	
\end{equation*}
Then~\eqref{eq:d2} implies  that
\begin{align*}
&\langle A(u)-A(v),u-v \rangle \\
&\leq -|u-v|_{W_0^{1,2}}^2+C|u-v|_{W_0^{1,2}}|u-v|_{L^2} 
+C|v|_{W_0^{1,2}}|u-v|_{L^2}|u-v|_{W_0^{1,2}} \\
&\quad \quad  +C|u-v|_{L^2}^2 + C|v|_{L^{4}}^2|u-v|_{W_0^{1,2}}|u-v|_{L^2}.
\end{align*}
Young's inequality with some $\epsilon>0$ finally 
leads to 
\begin{equation}
\label{eq:exlmon for A}
\begin{split}
\langle A(u) & -A(v),u-v \rangle \\ 
& \leq (\epsilon-1)|u-v|^2_{W_0^{1,2}}
+C(1+|v|_{W_0^{1,2}}^2 + |v|_{L^{4}}^{4})|u-v|_{L^2}^2. 
\end{split}	
\end{equation}
Moreover,
\[
\sum_{i=1}^d|B^i(u)-B^i(v)|_{L^2}^2
\leq 2\gamma^2 |u-v|_{W_0^{1,2}}^2+C|u-v|_{L^2}^2 \,.
\]
Thus using~\eqref{eq:d2} once again, one obtains
\begin{equation*}
\begin{split}
2 \langle & A(u)-A(v),u-v \rangle + \sum_{i=1}^d|B^i(u)-B^i(v)|_{L^2}^2 \\
& \leq (2\epsilon + 2\gamma^2 - 2)|u-v|^2_{W_0^{1,2}}
+C(1+|v|_{W_0^{1,2}}^2 + |v|_{L^2}^2 |v|_{W_0^{1,2}}^2)|u-v|_{L^2}^2\,.
\end{split}	
\end{equation*}
If $\gamma \in (-1,1)$, then one can get that for some $\theta > 0$,
\begin{equation*}
\begin{split}
2 \langle & A(u)-A(v),u-v \rangle + \sum_{i=1}^d|B^i(u)-B^i(v)|_{L^2}^2 
+\theta|u-v|^2_{W_0^{1,2}} \\
& \leq C(1+|v|_{W_0^{1,2}}^2)(1 + |v|_{L^2}^2)|u-v|_{L^2}^2,
\end{split}	
\end{equation*}
for all $u,v\in W_0^{1,2}(\mathscr{D})$.
Hence Assumption A-\ref{ass:lmon} is satisfied with $\alpha:=2$ and $\beta := 2$.

The next condition that ought to be verified is coercivity.
Taking $v=0$ in~\eqref{eq:exlmon for A}, one obtains for all $u\in W_0^{1,2}(\mathscr{D})$
\begin{align*}
\langle A(u),u \rangle \leq (\epsilon-1)|u|^2_{W_0^{1,2}}
+C|u|_{L^2}^2 
\end{align*}
which implies, together with the assumptions on $h$, that 
\begin{align*}
2 \langle A(u),u \rangle &+ (p_0-1)\sum_{i=1}^d|B^i(u)|_{L^2}^2 \\
&\leq \left(2\epsilon+2\gamma^2(p_0-1)-2\right)|u|^2_{W_0^{1,2}}+C\big(1+|u|_{L^2}^2\big).
\end{align*}
One can now take $p_0 := 4$ and see that if $\gamma^2 < 1/3$, then
Assumption A-\ref{ass:coer} holds with $\theta:=2-2\epsilon-6\gamma^2$
for $\epsilon > 0$ sufficiently small.

Finally one wishes to verify the growth condition. 
Using the boundedness of $g$ and H\"{o}lder's inequality one obtains,
for $u\in W_0^{1,2}(\mathscr{D})$, that
\[
|g(u)\nabla u|_{W^{-1,2}}\leq C |u|_{W_0^{1,2}}.
\]
Moreover, due to H\"older's inequality, one gets that for any $1\leq q < \infty$ 
and $u, v\in W_0^{1,2}(\mathscr{D})$
\[
\langle f(u), v\rangle \leq C|v|_{L^q} + C|u|_{L^{3\frac{q}{q-1}}}^3|v|_{L^q}
\leq C|v|_{W^{1,2}_0} + C|u|_{L^{3\frac{q}{q-1}}}^3|v|_{W^{1,2}_0},
\]
where the last inequality is consequence of the Sobolev embedding and the fact 
that $d=1$ or $2$. 
Hence, with $q=6$ one obtains that
\[
|f(u)|_{W^{-1,2}} \leq C\left(1+|u|_{L^{\frac{18}{5}}}^3\right) 
\leq C\left(1+|u|_{L^2}|u|_{L^6}^2\right), 
\]
where the last inequality follows from interpolation between spaces of integrable
functions, see e.g.~\cite[Theorem 1.24]{roubicek05}.
Finally, using the Sobolev embedding again, one can see that
\[
|A(u)|^2_{W^{-1,2}}\leq C\left(1+|u|_{W^{1,2}_0}^2\right)\left(1+|u|_{L^2}^2\right)
\]
thus Assumption A-\ref{ass:groA} is satisfied with $\alpha=2$, $\beta=2$.

If $d=1$ or $2$, $\alpha = 2$, $\beta = 2$, $p_0=4$, $\gamma^2 < 1/3$ 
and $u_0\in L^4(\Omega;L^2(\mathscr{D}))$ 
is $\mathscr{F}_0$-measurable then, 
in view of Theorems \ref{thm:apriori}, \ref{thm:unique} and \ref{thm:main},
one can conclude that equation~\eqref{ex:1} has a unique solution 
and moreover for any $p < 4$ one has
\[
\mathbb{E} \Big( \sup_{t \in [0,T]} |u_t|_{L^2}^p + \int_0^T |u_t|_{W_0^{1,2}}^2 dt \Big) < C \left(1+\mathbb{E}|u_0|_{L^2}^4\right).
\]
\end{example}

\begin{example}[Stochastic p-Laplace equation] 
\label{ex_quasi}

For $\alpha > 2$, consider the stochastic partial differential equation
\begin{equation}                                                 \label{ex:quasi}
\begin{split}
du_t  =   \Big(  \sum_{i=1}^d  D_i\big(|D_i u_t|^{\alpha-2}D_i u_t\big)  + & f(u_t) \Big)\,dt \\
& + \sum_{i=1}^d \gamma \big |D_i u_t|^\frac{\alpha}{2} \,dW_t^i + \sum_{i\in \mathbb{N}}h_i (u_t)\,dW_t^i  
\end{split}    
\end{equation}
on $(0,T)\times \mathscr{D}$, where $u_t = 0$ on $\partial \mathscr{D}$ and 
$u_0$ is a given $\mathcal{F}_0$-measurable random variable. Moreover $W^i$ are independent Wiener processes. Further, assume that there are constants $r,s,t \geq 1$ and continuous function $f $ on $\mathbb{R}$ such that
\begin{equation*}
\begin{split}
& \qquad f(x)x \leq  K(1+|x|^{\frac{\alpha}{2}+1}); \,\,\, |f(x)|\leq K(1+|x|^r) \\
& \text{ and }\,\, 
(f(x)-f(y))(x-y) \leq K(1+|y|^s)|x-y|^t\,\,\, \forall x,y\in \mathbb{R}\,.
\end{split}
\end{equation*} 
Finally, for $i \in \mathbb{N}$, let $h_i:\mathbb{R}\to \mathbb{R}$ 
be Lipschitz continuous with Lipschitz constants $M_i$ such that  the sequence $(M_i)_{i\in\mathbb N} \in \ell^2$.  
Let $A:W^{1,\alpha}_0(\mathscr{D}) \to W^{-1,\alpha}(\mathscr{D})$ be given by 
 \[
A(u):= \sum_{i=1}^d D_i \big(|D_i u|^{\alpha-2}D_i u\big) + f(u) 
\]
and $B^i:W^{1,\alpha}_0(\mathscr{D}) \to L^2(\mathscr{D})$  be given by,
\begin{equation*}
 B^i(u):= 
 \begin{cases}
 \gamma |D_iu|^\frac{\alpha}{2}+h_i(u) & \text{for} \ i=1,2,\ldots,d, \\
 h_i(u) & \text{otherwise} \,.
 \end{cases}
  \end{equation*}
%
It will now be shown that these operators satisfy Assumptions A-\ref{ass:hem} to A-\ref{ass:groA} if any of the following holds:

(1.) $d<\alpha,\, r=\alpha+1,\, s\leq \alpha, \, t=2$ and $u_0\in L^6(\Omega;L^2(\mathscr{D}))$.

(2.) $d>\alpha,\, r=\frac{2\alpha}{d}+\alpha-1, \, s\leq \min \Big \{\frac{\alpha^2(t-2)}{(d-\alpha)(\alpha-2)}, \frac{\alpha(\alpha-t)}{(\alpha-2)}\Big\},\, 2<t<\alpha$ and $u_0\in L^6(\Omega;L^2(\mathscr{D}))$.

\textbf{Case (1.)} One immediately notices that A-\ref{ass:hem} holds since $f$ is continuous. 

One now wishes to verify the local monotonicity condition. From standard calculations for $p$-Laplace operators, one obtains for each $i=1,2,\ldots,d$,
\[
\big \langle  D_i \big(|D_ i u|^{\alpha-2}D_i u \big)- D_i \big(|D_i v|^{\alpha-2}D_i v \big),u-v \big\rangle + \big|\gamma |D_i u|^\frac{\alpha}{2}-\gamma|D_i v|^\frac{\alpha}{2}\big|_{L^2}^2 \leq 0
\]
provided $\gamma^2 \leq \frac{4(\alpha-1)}{\alpha^2}$.
Further for $d<\alpha$, one has $W_0^{1,\alpha}(\mathscr{D})\subset L^\infty(\mathscr{D}) $ by the Sobolev embedding and therefore using the assumptions imposed on $f$ one obtains that
for $u,v\in W^{1,\alpha}_0(\mathscr{D})$, 
\begin{equation*}
\begin{split}
& \langle f(u)-f(v), u-v \rangle \leq K \int_{\mathscr{D}} (1+|v(x)|^s)|u(x)-v(x)|^2 dx \\
& \leq K (1+|v|^s_{L^\infty})|u-v|_{L^2}^2 \leq C (1+|v|^s_{W_0^{1,\alpha}})|u-v|_{L^2}^2  \leq C (1+|v|^{\alpha}_{W_0^{1,\alpha}})|u-v|_{L^2}^2 
\end{split}
\end{equation*}
if $s\leq \alpha$.
Moreover, using Lipschitz continuity of the functions $h_i$ 
\begin{equation*}
|h^i(u)-h^i(v)|_{L^2}^2 \leq M_i^2|u-v|_{L^2}^2 \, .
\end{equation*}
Therefore, for all $u,v\in W^{1,\alpha}_0(\mathscr{D})$
\begin{equation*}
\begin{split}
& 2  \langle A(u)-A(v),u-v \rangle + \sum_{i\in \mathbb{N}}|B^i(u)-B^i(v)|_{L^2}^2 \\
 & \leq 2  \sum_{i=1}^d \Big[\big \langle  D_i \big(|D_ i u|^{\alpha-2}D_i u \big)- D_i \big(|D_i v|^{\alpha-2}D_i v \big),u-v \big\rangle + \big|\gamma |D_i u|^\frac{\alpha}{2}-\gamma|D_i v|^\frac{\alpha}{2}\big|_{L^2}^2 \Big] \\
  & \quad + 2\Big[  \langle f(u)-f(v), u-v \rangle + \sum_{i\in \mathbb{N}}|h^i(u)-h^i(v)|_{L^2}^2  \Big]    \leq C \Big(1+|v|^{\alpha}_{W_0^{1,\alpha}}\Big)|u-v|_{L^2}^2 \, .
\end{split}	
\end{equation*}

Hence Assumption A-\ref{ass:lmon} is satisfied with $\beta := 0$.
Again,
\begin{equation*}
\begin{split}
2 \sum_{i=1}^d \big\langle & D_i\big(|D_i u|^{\alpha-2}D_i u \big),u \big\rangle  
=-2 |u|_{W_0^{1,\alpha}}^{\alpha} \, .
\end{split}
\end{equation*}
Using assumptions on $f$, Holder's inequality and Sobolev embedding as above, one obtains
\begin{equation*}
\begin{split}
2  \langle f(u),u \rangle & \leq 2 K \int_{\mathscr{D}} (1+|u(x)|^{\frac{\alpha}{2}+1})dx \leq 2 K(1+|u|_{L^\infty}^\frac{\alpha}{2}|u|_{L^2}) \\
& \leq C(1+|u|_{W_0^{1,\alpha}}^\frac{\alpha}{2}|u|_{L^2}) \leq \delta |u|_{W_0^{1,\alpha}}^{\alpha}+C(1+|u|_{L^2}^2)
\end{split}
\end{equation*}
where last inequality is obtained using Young's inequality with sufficiently small $\delta >0$. Further, for any $p_0>2$
\[
(p_0-1)\sum_{i=1}^d 2\big|\gamma |D_i u|^\frac{\alpha}{2}\big|_{L^2}^2=(p_0-1)2\gamma^2 \sum_{i=1}^d\int_{\mathscr{D}} |D_i u(x)|^{\alpha} dx=(p_0-1)2\gamma^2 |u|_{W_0^{1,\alpha}}^{\alpha} \,.
\]
Hence Assumptions A-\ref{ass:coer} is satisfied with $\theta := 2-2(p_0-1)\gamma^2-\delta>0$.
Note that using H\"older's inequality, one gets for $u,v \in W^{1,\alpha}_0(\mathscr{D})$,
\begin{equation*}
\begin{split}
\int_{\mathscr{D}} |D_i u(x)|^{\alpha-1} |D_i v(x)|dx & \leq \Big(\int_{\mathscr{D}} |D_i u(x)|^{\alpha} dx\Big)^\frac{\alpha-1}{\alpha}\Big(\int_{\mathscr{D}} |D_i v(x)|^\alpha dx\Big)^\frac{1}{\alpha} \\
& \leq \Big(\sum_{i=1}^d\int_{\mathscr{D}} |D_i u(x)|^{\alpha} dx\Big)^\frac{\alpha-1}{\alpha}\Big(\sum_{i=1}^d\int_{\mathscr{D}} |D_i v(x)|^\alpha dx\Big)^\frac{1}{\alpha} \\
& = |u|_{W_0^{1,\alpha}}^{\alpha-1} |v|_{W_0^{1,\alpha}} \, .
\end{split}
\end{equation*}
Further using assumption on $f$ taking $r=\alpha+1$, H\"older's inequality, Gagliardo--Nirenberg inequality and Sobolev embedding,
\begin{equation*}
\begin{split}
& \int_{\mathscr{D}} |f(u(x))||v(x)|dx \leq K \int_{\mathscr{D}} \big(1+|u(x)|^{\alpha+1}\big)|v(x)|dx \\
& \leq K |v|_{L^\infty}(1+|u|_{L^{\alpha+1}}^{\alpha+1}) \leq K |v|_{W_0^{1,\alpha}}(1+ |u|_{L^\infty}^{\alpha-1}|u|_{L^2}^2) \leq K |v|_{W_0^{1,\alpha}}(1+ |u|_{W_0^{1,\alpha}}^{\alpha-1}|u|_{L^2}^2)
\end{split}
\end{equation*}
and hence
\[
|A(u)|_{W^{-1,\alpha}} \leq K |u|_{W_0^{1,\alpha}}^{\alpha-1} + K(1+ |u|_{W_0^{1,\alpha}}^{\alpha-1}|u|_{L^2}^2) \leq K(1+ |u|_{W_0^{1,\alpha}}^{\alpha-1})(1+|u|_{L^2}^2) \, .
\]
Thus Assumption A-\ref{ass:groA} holds with $\beta=\frac{2\alpha}{\alpha-1}<4$ and in view of Theorems \ref{thm:apriori}, \ref{thm:unique} and \ref{thm:main},
one can conclude that equation~\eqref{ex:quasi} has a unique solution 
and moreover for any $p < \frac{4\alpha-2}{\alpha-1}$ one has
\[
\mathbb{E} \Big( \sup_{t \in [0,T]} |u_t|_{L^2}^p + \int_0^T |u_t|_{W_0^{1,\alpha}}^\alpha dt \Big) < C \left(1+\mathbb{E}|u_0|_{L^2}^\frac{4\alpha-2}{\alpha-1}\right).
\]
\textbf{Case (2.)} In the case $d>\alpha$, one can obtain the result using the Sobolev embedding $W_0^{1,\alpha}(\mathscr{D})\subset L^\frac{d\alpha}{d-\alpha}(\mathscr{D})$ and following the same steps as in \cite{brz14}.
\end{example}

\begin{example}[Stochastic Burgers equation]
\label{ex2}
Let $d=1$ and $\mathscr{D}=(0,1)$. 
Let $\gamma\in (-\sqrt{1/3},\sqrt{1/3})$ be a constant
and let $h:\mathbb{R} \to \mathbb{R}$ be Lipschitz continuous.
Consider the stochastic partial differential equation
\begin{align}
du_t= \Big(\Delta u_t+u_tDu_t \Big) dt+\big(\gamma Du_t+h(u_t)\big)dW_t \,\,\,\text{on $(0,T)\times \mathscr{D}$}, \label{ex:2}
\end{align}
where $u_t = 0$ on $\partial \mathscr{D}$ 
and an $L^2(\mathscr{D})$-valued, $\mathcal{F}_0$-measurable $u_0$ is a given initial condition.
Here $W$ is a real-valued Wiener process.
Weak formulation of this equation can be interpreted as a stochastic evolution
equation as follows.

Define $A:W_0^{1,2}(\mathscr{D})\to W^{-1,2}(\mathscr{D})$ and $B:W_0^{1,2}(\mathscr{D}) \to L^2(\mathscr{D})$ as
\[
A(u):=\Delta u+uDu\,\,\, \text{ and }\,\,\, B(u):=\gamma Du+h(u).
\]
Note that Assumption A-\ref{ass:hem} is satisfied following the same arguments as in Example~\ref{ex1}. 
Next, one would like to check the local monotonicity assumption.
Note that, if $u,v \in W^{1,2}_0(\mathscr{D})$, then 
\[
\frac{1}{2}D\left[u^2-v^2\right] = uDu - vDv
\]
and so using integration by parts 
\[
\langle uDu - vDv, u-v\rangle = -\frac{1}{2}\langle u^2-v^2,D(u-v) \rangle\,.
\]
Thus,
\[
\begin{split}
 \langle A(u)-A(v), & u-v \rangle \\
& = -|u-v|_{W_0^{1,2}}^2 -\frac{1}{2} \langle (u-v)^2, D(u-v) \rangle 
- \langle v(u-v),D(u-v)\rangle. 
\end{split}
\]
But using integration by parts again we see that $\langle (u-v)^2, D(u-v) \rangle = 0$ and so
\[
\langle A(u)-A(v),u-v \rangle 
 = -|u-v|_{W_0^{1,2}}^2 - \langle v(u-v),D(u-v)\rangle. 
\]
So from H\"{o}lder's inequality one observes that 
\begin{equation*}
\langle A(u)-A(v),u-v \rangle  
\leq 
-|u-v|_{W_0^{1,2}}^2+ |v|_{L^4}|u-v|_{L^4}|u-v|_{W_0^{1,2}}
\end{equation*}
and thus Gagliardo--Nirenberg inequality, see~\eqref{eq:d2}, 
and Young's inequality imply that for any $\epsilon > 0$ 
\begin{equation}
\label{eq:ex2mon}
\langle A(u)-A(v),u-v \rangle  
\leq 
 -|u-v|_{W_0^{1,2}}^2+ \epsilon |u-v|_{W_0^{1,2}}^2+ C|v|^2_{L^2}|v|^2_{W_0^{1,2}}|u-v|_{L^2}^2. 
\end{equation}
This, along with Lipschitz continuity of $h$, gives
\begin{align*}
2 \langle & A(u)-A(v),u-v \rangle + |B(u)-B(v)|_{L^2}^2 \\
&\leq (-2+2\epsilon+2\gamma^2)|u-v|^2_{W_0^{1,2}}+C(1+|v|_{L^2}^2)(1+|v|_{W_0^{1,2}}^2)|u-v|_{L^2}^2 
\end{align*}
for all $u,v\in W_0^{1,2}(\mathscr{D})$.
As $\gamma^2 \in (0,1/3)$ one can take $\epsilon > 0$ sufficiently small
so that $-1+\epsilon+\gamma^2 < 0$ 
and hence Assumption A-\ref{ass:lmon} is satisfied with $\alpha:=~2$ and $\beta:=2$. 

The next step is to show that the coercivity assumption holds with $p_0=4$.
Indeed, substituting $v=0$ in \eqref{eq:ex2mon}, one obtains
\begin{align}
\langle A(u),u \rangle \leq (-1+\epsilon)|u|^2_{W_0^{1,2}} \notag
\end{align}
which along with linear growth of $h$ implies that
\[
2 \langle A(u),u \rangle + 3|B(u)|_{L^2}^2 
\leq (-2+2\epsilon+6\gamma^2)|u|^2_{W_0^{1,2}}+C\big(1+|u|_{L^2}^2\big). 
\]
Note that since $\gamma^2 \in (0,1/3)$ one can take $\epsilon > 0$ sufficiently small
so that $\theta:=2-2\epsilon-6\gamma^2 > 0$. Then with $f := C$, Assumption A-\ref{ass:coer} holds.

Finally, one should verify the growth assumption on $A$. 
Using integration by parts, H\"{o}lder's inequality and \eqref{eq:d2} one obtains for $u,v\in W_0^{1,2}(\mathscr{D})$,
\begin{align}
&\langle u Du,v\rangle=-\frac{1}{2} \langle u,Dv \rangle 
\leq 
\frac{1}{2}|u|_{L^4}^2 |v|_{W_0^{1,2}}  \leq C |u|_{L^2} |u|_{W_0^{1,2}} |v|_{W_0^{1,2}}\notag
\end{align}
which then implies that
\begin{align} \label{eq:ex2growth}
|uDu|_{W^{-1,2}}\leq C |u|_{L^2}|u|_{W_0^{1,2}}.
\end{align} 
Hence using \eqref{eq:delta}, one obtains for all $u\in W_0^{1,2}(\mathscr{D})$
\begin{align*}
|A(u)|^2_{W^{-1,2}}\leq C|u|^2_{W_0^{1,2}} (1+|u|_{L^2}^2)
\end{align*}
proving that Assumption A-\ref{ass:groA} is  satisfied for $\alpha=2,\beta=2$ and $f = C$.

Thus, in view of Theorems \ref{thm:apriori}, \ref{thm:unique} and \ref{thm:main}, if $u_0\in L^4(\Omega;L^2(\mathscr{D}))$, then equation \eqref{ex:2} has a unique solution $(u_t)_{t\in[0,T]}$ and for any $p<4$
\[
\mathbb{E} \Big( \sup_{t \in [0,T]} |u_t|_{L^2}^p + \int_0^T |u_t|_{W_0^{1,2}}^2 dt \Big) < C\left(1 + \mathbb{E}|u_0|_{L^2}^{4} \right),
\]
where we recall in particular that $C$ depends on $T$.
\end{example}

\begin{example}[]
\label{ex2'}
Let $d=1$ and $\mathscr{D}=(0,1)$. 
Let $\gamma\in (-\sqrt{2/5},\sqrt{2/5})$ be a constant.
Consider the stochastic partial differential equation
\begin{align}
du_t= \Big(\Delta u_t+u_tDu_t-u_t^3 \Big) dt+ \gamma u_t^2dW_t \,\,\,\text{on $(0,T)\times \mathscr{D}$}, \label{ex:2'}
\end{align}
where $u_t = 0$ on $\partial \mathscr{D}$ 
and an $L^2(\mathscr{D})$-valued, $\mathcal{F}_0$-measurable $u_0$ is a given initial condition.
Here $W$ is a real-valued Wiener process.
Weak formulation of this equation can be interpreted as a stochastic evolution
equation as follows.

Define $A:W_0^{1,2}(\mathscr{D})\to W^{-1,2}(\mathscr{D})$ and $B:W_0^{1,2}(\mathscr{D}) \to L^2(\mathscr{D})$ as
\[
A(u):=\Delta u+uDu-u^3\,\,\, \text{ and }\,\,\, B(u):=\gamma u^2.
\]
where, $A$ and $B$ are well-defined using the Sobolev embedding $W_0^{1,2}(\mathscr{D}) \subset L^\infty(\mathscr{D})$ and \eqref{eq:d2} above.
Clearly, Assumption A-\ref{ass:hem} is satisfied.
Further, using Mean value theorem it is easy to observe that
\[
\langle -u^3 + v^3, u-v \rangle +|\gamma(u^2-v^2)|_{L^2}^2 \leq 0
\]  
since $\gamma^2<2/5$ and hence using \eqref{eq:ex2mon}, one obtains 
\begin{align*}
2 \langle & A(u)-A(v),u-v \rangle + |B(u)-B(v)|_{L^2}^2 \\
&\leq (-2+2\epsilon)|u-v|^2_{W_0^{1,2}}+C(1+|v|_{L^2}^2)(1+|v|_{W_0^{1,2}}^2)|u-v|_{L^2}^2 \\
&\leq (-2+2\epsilon)|u-v|^2_{W_0^{1,2}}+C(1+|v|_{L^2}^4)(1+|v|_{W_0^{1,2}}^2)|u-v|_{L^2}^2
\end{align*}
for any $\epsilon > 0$ and for all $u,v\in W_0^{1,2}(\mathscr{D})$.
By choosing $0<\epsilon <1$,  Assumption A-\ref{ass:lmon} is satisfied with $\alpha:=~2$ and $\beta:=4$. 

Further Assumption A-\ref{ass:coer} holds with $p_0=6$ and $\theta=2-2\epsilon$.
Indeed, one has
\[
2 \langle A(u),u \rangle + 5|B(u)|_{L^2}^2 
\leq (-2+2\epsilon)|u|^2_{W_0^{1,2}}. 
\]

Finally, one should verify the growth assumption on $A$. 
Using Sobolev embedding one obtains for $u,v\in W_0^{1,2}(\mathscr{D})$,
\begin{align}
& |\langle -u^3,v\rangle| \leq |u|_\infty |v|_\infty|u|_{L^2}^2  \leq C |u|_{W_0^{1,2}} |v|_{W_0^{1,2}} |u|_{L^2}^2 \notag
\end{align}
which then implies that
\begin{align*}
|-u^3|_{W^{-1,2}}\leq C |u|_{L^2}^2 |u|_{W_0^{1,2}}.
\end{align*}
Hence using \eqref{eq:delta} \eqref{eq:ex2growth} one obtains for all $u\in W_0^{1,2}(\mathscr{D})$
\begin{align*}
|A(u)|^2_{W^{-1,2}}\leq C|u|^2_{W_0^{1,2}} (1+|u|_{L^2}^4)
\end{align*}
proving that Assumption A-\ref{ass:groA} is  satisfied for $\alpha=2,\beta=4$ and $f = C$.

Thus, in view of Theorems \ref{thm:apriori}, \ref{thm:unique} and \ref{thm:main}, if $u_0\in L^6(\Omega;L^2(\mathscr{D}))$, then equation \eqref{ex:2'} has a unique solution $(u_t)_{t\in[0,T]}$ and for any $p<6$
\[
\mathbb{E} \Big( \sup_{t \in [0,T]} |u_t|_{L^2}^p + \int_0^T |u_t|_{W_0^{1,2}}^2 dt \Big) < C\left(1 + \mathbb{E}|u_0|_{L^2}^{6} \right),
\]
where we recall in particular that $C$ depends on $T$.
\end{example}

\begin{remark}
\label{rem gamma_brez}
Note that taking $h=0$ in previous examples, one requires $\gamma^2 < 2/3$ 
in Examples~\ref{ex1}, \ref{ex2} and less than $\frac{8(\alpha-1)}{\alpha^2} \wedge\frac{2(\alpha-1)}{3\alpha-1}$ 
in Example~\ref{ex_quasi}. 
Here, $\gamma^2$ is the coefficient of $|v|_V^\alpha$ appearing in the growth of the operator $B$. 
However, the corresponding values required in main theorem of \cite{brz14} would be less than $2/5$ for Examples~\ref{ex1}, \ref{ex2} and less than $\frac{8(\alpha-1)}{\alpha^2} \wedge\frac{2(\alpha-1)}{5\alpha-1}$ for Example~\ref{ex_quasi}. 
Thus, the restriction on $\gamma$ appearing in the growth assumption of operator $B$ is not optimal in \cite{brz14}. 
Further, operators $B$ having growth like in Example~\ref{ex2'} cannot be covered by \cite{brz14}. 
\end{remark}

One should note that the restriction on the range of values $\gamma$ may take is not surprising
in view of known results for linear stochastic partial 
differential equations where the ``stochastic parabolicity'' condition is needed.
To see how this arises, consider the initial value problem 
\begin{equation*}
dv_t =(1-\frac{1}{2}\gamma ^2)\Delta v_t \, dt \,\,\, \text{on}\,\,\, (0,T)\times \mathbb{R}^d
\end{equation*}
with $v_0\in L^2(\mathbb{R}^d)$ given as an initial value.
This is well-posed if $(1-\frac{1}{2}\gamma^2)>0$. 
Let $u_t(x):=v(t,x+\gamma W_t)$, where $W$ is $\mathbb{R}$-valued Wiener process.
It\^o's formula implies that 
\begin{equation*}
du_t=\Delta u_t dt+\sum_{i=1}^d \gamma D_i u_t dW_t, \,\,\, \text{on}\,\,\, (0,T)\times \mathbb{R}^d, \,\,\, u_0=v_0. 
\end{equation*}
Hence one can only reasonably expect this stochastic partial differential equation to be well-posed if $(1-\frac{1}{2}\gamma^2)>0$.

On the other hand, one can see that the range of values of $\gamma$ one may 
take, so that Assumption A-\ref{ass:coer} is satisfied, depends on $p_0$.
This may seem surprising in view of results in Krylov~\cite{krylov99}
on $L^p$-theory for stochastic partial differential equations.
The following example, which is not covered in~\cite{krylov99}, from Brze\'{z}niak and Veraar~\cite{brz12}, 
explores this question further.

\begin{example}
\label{ex sharp}
Consider the stochastic partial differential equation
\begin{equation}
du_t= \Delta u_t dt+2 \gamma (-\Delta)^\frac{1}{2} u_t \,dW_t 
\,\,\, \text{ on } \,\,\, (0,T)\times \mathbb{T}, \label{ex:3}
\end{equation}
where $\mathbb{T}$ is the one-dimensional torus $\mathbb{R}/(2\pi\mathbb{Z})$,
$\gamma\in \mathbb{R}$ is a constant 
and $\mathcal{F}_0$-measurable $u_0$ is a given initial condition.
Here $W$ is a real-valued Wiener process.

For $\gamma^2 \in (0,1/2)$ and $u_0 \in L^2(\Omega; L^2(\mathbb{T}))$ 
the results in Krylov and Rozovskii~\cite{krylov81} imply existence and uniqueness 
of the solution to~\eqref{ex:3} and moreover the solution satisfies
\begin{equation}
\mathbb{E}  \sup_{t \in [0,T]} |u_t|_{L^2(\mathbb{T})}^2 
< C\mathbb{E}\Big(1+|u_0|_{L^2(\mathbb{T})}^2\Big) . \notag
\end{equation}
On the other hand Brze\'{z}niak and Veraar~\cite{brz12} have shown that if 
\[
2\gamma^2 (p-1)>1,
\]
then the problem \eqref{ex:3} is not well-posed in $L^p((0,T)\times\Omega;L^2(\mathbb{T}))$.
It will be shown that this example fits in the framework considered in this paper
and that the coercivity condition, Assumption A-\ref{ass:coer}, is satisfied 
as long as 
\begin{equation}
\label{eq b and p0 for coer}
2\gamma^2 (p_0-1) < 1.	
\end{equation}
This shows that the coercivity condition in this paper is sharp, since~\eqref{ex:3} is ill-posed as soon as Assumption A-\ref{ass:coer} 
does not hold.

Let the space $L^2(\mathbb{T})$ denote the Lebesgue space of equivalence classes 
of $\mathbb{C}$-valued measurable functions $u$ defined on any interval of length $2\pi$, which are $2\pi$-periodic and the norm 
\[
|u|_{L^2(\mathbb{T})}:=\Big(\int_\mathbb{T}|u(x)|^2 dx\Big)^\frac{1}{2} < \infty\,.
\]
Further, $W^{1,2}(\mathbb{T})$ denotes the closure of $C^\infty(\mathbb{T})$, the space of smooth functions, in $L^2(\mathbb{T})$ with respect to 
the norm $|\cdot|_{1,2}$ given by
\[
|u|_{1,2}:=\bigg(\int_\mathbb{T}\big(|u(x)|^2+|Du(x)|^2  \big)dx \bigg)^\frac{1}{2}. 
\]
Let $\mathcal{F}: L^2(\mathbb{T})\to \ell^2(\mathbb{Z})$ be the Fourier transform given by
\[
\mathcal{F}u:=(\hat{u}_k)_{k\in\mathbb{Z}} \,\,\text{with}\,\,\hat{u}_k=\frac{1}{\sqrt{2\pi}}\int_\mathbb{T}u(x)e^{-ikx}dx                                                \]
and $\mathcal{F}^{-1}:\ell^2(\mathbb{Z})\to L^2(\mathbb{T})$ its inverse
which is given by
\[
\mathcal{F}^{-1}(\hat{u}_k)_{k\in\mathbb{Z}}=:u\,\,\text{with}\,\,u(x)=\frac{1}{\sqrt{2\pi}}\sum_{k\in \mathbb{Z}}\hat{u}_ke^{ikx}.
\]
For $u\in W^{1,2}(\mathbb{T})$, one has
\begin{equation}
\label{eq:plancherel}
|u|^2_{W^{1,2}(\mathbb{T})}=|\mathcal{F} u|^2_{\ell^2(\mathbb{Z})}+|\mathcal{F} (Du)|^2_{\ell^2(\mathbb{Z})}, \,\,\, \text{ since } \,\,\, |u|^2_{L^2(\mathbb{T})}=|\mathcal{F} u|^2_{\ell^2(\mathbb{Z})}.
\end{equation} 
Furthermore, for each $k\in\mathbb{Z}$,
\begin{align}
[\mathcal{F}(Du)](k) =ik(\mathcal{F}u)(k)\label{eq:fourier}.
\end{align}

Consider the operator $(-\Delta)^\frac{1}{2}:W^{1,2}(\mathbb{T})\to L^2(\mathbb{T})$ defined by
\begin{equation}
(-\Delta)^\frac{1}{2}u:=\mathcal{F}^{-1}\Big(\big(|k|(\mathcal{F}u)(k)\big)_{k\in\mathbb{Z}}\Big) \notag
\end{equation}
and the operators $A:W^{1,2}(\mathbb{T})\to W^{-1,2}(\mathbb{T})$ and $B:W^{1,2}(\mathbb{T})\to L^2(\mathbb{T})$ defined by
\[
A(u)=\Delta u \text{ and } B(u)=2 \gamma (-\Delta)^\frac{1}{2} u.
\]
It will be shown that these satisfy Assumptions A-\ref{ass:hem} to A-\ref{ass:groA}. 
Using the arguments given in Example $1$, the operator $A$ satisfies Assumptions A-\ref{ass:hem} and A-\ref{ass:groA} with $\alpha=2$, $\beta=0$, $p_0=2$ and $L=0$.
Then, using \eqref{eq:plancherel} and \eqref{eq:fourier}, one obtains
\begin{align}
2 \langle  A(u)-A(v),u-v \rangle  + & |B(u)-B(v)|_{L^2(\mathbb{T})}^2 \notag \\
& =(-2+4 \gamma^2)\sum_{k\in\mathbb{Z}}k^2\big|(\mathcal{F}u)(k)-(\mathcal{F}u)(k)\big|^2 \leq 0\notag
\end{align}
provided $2\gamma^2\leq 1$. Hence operators $A$ and $B$ satisfy Assumption A-\ref{ass:lmon} if $2\gamma^2\leq 1$. Furthermore, for any
 $\theta >0$ and 
$p_0\geq 2$, one obtains 
\begin{align*}
2 &\langle A(u),u \rangle +(p_0-1)|Bu|^2_{L^2(\mathbb{T})}+\theta|u|^2_{W^{1,2}(\mathbb{T})} \\
&= (4\gamma^2 (p_0-1)-2+\theta)\sum_{k\in\mathbb{Z}}k^2|(\mathcal{F}u)(k)|^2+\theta |u|_{L^2(\mathbb{T})}^2.
\end{align*} 
Note that there is $\theta > 0$ such that $(4\gamma^2 (p_0-1)-2+\theta)\leq 0$ if and only if $2\gamma^2 (p_0-1)<1$.
Hence Assumption A-\ref{ass:coer} holds if and only if~\eqref{eq b and p0 for coer} holds.

Thus from Theorem~\ref{thm:apriori} one can see that the solution satisfies
\begin{equation*}
\mathbb{E}  \sup_{t \in [0,T]} |u_t|_{L^2(\mathbb{T})}^p 
< C\mathbb{E}\Big(1+ |u_0|_{L^2(\mathbb{T})}^{p_0}\Big)
\end{equation*}
for $p\in [2,p_0)$ if $p_0>2$ and for $p=2$ otherwise.

\end{example}

\section*{Acknowledgement(s)}

The authors are indebted to Istv\'an Gy\"ongy for many useful discussions 
concerning problems arising in this paper.

\appendix

\section{Hilbert-Space Valued Wiener Process}
\label{sec:cyl}
Many authors consider stochastic evolution equations with respect to 
cylindrical $Q$-Wiener process $\mathcal{W}$ taking values in a
separable Hilbert space $(U, (\cdot,\cdot)_U, |\cdot|_U)$. Here $Q$ is a linear, symmetric, non-negative definite and bounded operator on $U$. For an overview of stochastic integrals with respect to Hilbert-space valued Wiener processes, one may refer to Dalang and Sardanyons \cite{dalang11} or \cite{rockner07}.
The operator under the stochastic integral would be taking
values in the space of Hilbert--Schmidt operators, denoted by $L_2(U,H)$.
The stochastic evolution equation considered is then written as 
\begin{equation}                                  
\label{eq:see-cyl}
u_t=u_0+\int_0^tA_s(u_s)ds+\int_0^t B_s(u_s)d\mathcal{W}_s, \quad  t \in [0,T]\,,
\end{equation}
instead of~\eqref{eq:see}.
The aim of this section is to show that these formulations are equivalent.

First we show that the stochastic It\^o integral
with respect to cylindrical $Q$-Wiener process on a separable Hilbert space 
can be expressed in the form of infinite sum of stochastic It\^o integrals with respect to independent one-dimensional Wiener processes as considered in \eqref{eq:see}. 
Here $\mathcal{W}$ is cylindrical $Q$-Wiener process in $U$ with $Q=I$, the identity on $U$. 
Let $(u^j)_{j\in\mathbb{N}}$ be an orthonormal basis of $U$, which in this case are also the eigenvectors of $Q$  
corresponding to the eigenvalues $(\lambda^j)_{j\in\mathbb{N}}$ where $\lambda^j=1$ for each $j\in \mathbb{N}$.

For $t\in[0,T]$ and $j\in \mathbb{N}$, define $W^j_t:=( \mathcal{W}_t,u^j)_U$. Then it can be seen that the processes $(W_t^j)_{t\in [0,T]}, j\in \mathbb{N}$ are independent real-valued Wiener processes. However, the series $\sum _{j=1}^\infty W_t^ju^j =\sum _{j=1}^\infty\sqrt{\lambda^j}W_t^ju^j$ does not converge
in $L^2(\Omega;U)$ as $\sum _{j=1}^\infty \lambda^j$, i.e. trace of $Q$, is not finite.
Consider the linear map $J:U\to U$ given by 
\[
J u := \sum_{j=1}^\infty\frac{1}{j}(u, u^j)_U u^j \quad \forall u\in U.
\]
Note that $J u^j = \frac{1}{j}u^j$ for each $j\in\mathbb{N}$.
It can then be seen that $J$ is an injective mapping satisfying 
\[
\sum _{j=1}^\infty|Ju^j|_U^2 < \infty
\]
and for each $t\in [0,T]$, the series 
\[
\sum _{j=1}^\infty(\mathcal{W}_t,u^j)_UJu^j= \sum _{j=1}^\infty W_t^j  Ju^j
\]
converges in $L^2(\Omega;U)$. In fact, the series converges in $L^2(\Omega; C([0,T];U))$ and defines a $Q_1$-Wiener process on $U$ where $Q_1:=JJ^*$ is a bounded linear operator on $U$ which is nonnegative definite, symmetric and has finite trace. Moreover, $Q_1^\frac{1}{2}(U)=J(U)$ and $(Ju^j)_{j\in\mathbb{N}}$ forms an orthonormal basis of $J(U)$ where the norm on the space $Q_1^\frac{1}{2}(U)=J(U)$ is given by 
\[
|Ju|_{Q_1^\frac{1}{2}(U)}=|Q_1^{-\frac{1}{2}}Ju|_U=|u|_U \quad \forall u\in U.
\]
For details, one may refer to Proposition 2.5.2 in \cite{rockner07}.

Next we show that the two formulations of stochastic integral (with respect
to cylindrical $Q$-Wiener process, or written as an infinite sum) are equivalent.
Consider a progressively measurable process  $(B_t)_{t\in[0,T]}$ taking values in $ L_2(U;H)$, where $L_2(U;H)$ is the space of Hilbert Schmidt operators from $U$ to $H$. 
Note that 
\[
B_t(\omega)\in L_2(U;H) \iff B_t(\omega)\circ J^{-1}\in L_2(J(U);H)=L_2(Q_1^{\frac{1}{2}}(U);H)
\]
and then the stochastic integral with respect to cylindrical $Q$-Wiener processes is defined by the following 
\[
\int_0^t B_s\,d\mathcal{W}_s :=\int_0^t B_s \circ J^{-1}\,d\mathcal{W}_s\,, \quad t\in[0,T]
\]
where the integral on right-hand-side is with respect to $Q_1$-Wiener process on $U$ (see e.g.  Section 2.5.2 in \cite{rockner07}). 

Now we show that the above stochastic integral with respect to a cylindrical Wiener process can be expressed as an infinite sum of stochastic integrals of suitable $H$-valued processes with respect to independent real-valued Wiener processes. 
Define $B_t^j:=B_t(u^j)=(B_t\circ J^{-1})(Ju^j)$ for all $t\in [0,T]$ and $j\in\mathbb{N}$. Then $(B_t^j)_{j\in\mathbb{N}}\in \ell^2(H)$ since $B_t\in L_2(U;H)$. Further for $u\in U$, we have
\[
B_t(u)=(B_t\circ J^{-1})(Ju)=\sum _{j=1}^\infty(u,u^j)_U(B_t\circ J^{-1})(Ju^j)=\sum _{j=1}^\infty(u,u^j)_UB_t^j
\]
and hence
\begin{equation}
\label{eq appendix stoch ints}	
\int_0^t B_sd\mathcal{W}_s=\int_0^t B_s\circ J^{-1}d\mathcal{W}_s=\sum _{j=1}^\infty\int_0^t B_s^j\,dW_s^j\,.
\end{equation}
Thus~\eqref{eq appendix stoch ints} implies that $u$ is a solution to~\eqref{eq:see} if and only if it is a solution 
to~\eqref{eq:see-cyl}.

{\color{white}
\begin{assumptiontilde}
\end{assumptiontilde}
}
\vspace{-0.8cm}

Moreover assumptions in this paper made on operators $B^j:[0,T]\times\Omega\times V \to H$ can be equivalently replaced by assumptions on the operator $B:[0,T]\times\Omega\times V \to L_2(U;H)$ as follows.
Assumption A-\ref{ass:lmon}
can be equivalently replaced by:
\begin{assumptiontilde}[Local Monotonicity]\label{ass:lmon-cyl} 
Almost surely for all $t\in[0,T]$ and $x,\bar{x} \in V$,
\begin{align}
2\langle A_t(x)-A_t(\bar{x})&, x-\bar{x} \rangle + |B_t(x)-B_t(\bar{x})|_{L_2(U,H)}^2 \notag \\
& \leq  L(1+|\bar{x}|^\alpha_V)(1+|\bar{x}|^\beta_H)|x-\bar{x}|^2_H.\notag
\end{align}
\end{assumptiontilde}
Finally A-\ref{ass:coer} can be equivalently replaced by:
\begin{assumptiontilde}[Coercivity]\label{ass:coer-cyl} 
Almost surely for all $t\in[0,T]$ and $x \in V$,
$$2\langle A_t(x), x \rangle +(p_0-1)|B_t(x)|_{L_2(U,H)}^2 + \theta |x|_V^\alpha \leq f_t+K|x|^2_H.$$
\end{assumptiontilde}

\section{A Compactness Result}
\label{sec:auxiliary}
%
%
The following lemma is not new and is included for the convenience of reader. It allows one to obtain weakly-star convergent 
subsequences, under appropriate assumptions.

\begin{lemma} \label{lem:weak_star}
Let $X$ be a separable Banach space with dual $X^*$ and $\langle \cdot, \cdot \rangle$ denotes the duality pairing between $X$ and $X^*$. If $(S,\Sigma,\mu)$ is a measure space with $\mu(S)<~\infty$,  and $(u_n)_{n\in\mathbb{N}}$ is a sequence satisfying 
\begin{equation}
\sup_n\int_S|u_n|_{X^*}^p d\mu <\infty \label{eq:uniform_bound}
\end{equation}
for some $p\geq 2$, then there exists a subsequence $(n_k)$ and $u \in L^p(S,X^*)$ such that 
$(u_{n_k})$ converges weakly-star to $u$ as $n_k\to \infty$, i.e., 
\[
\int_S\langle u_{n_k},\varphi\rangle  d\mu  \to \int_S\langle u,\varphi\rangle  d\mu \quad\quad \forall \,\varphi\in L^\frac{p}{p-1}(S,X).
\]
\end{lemma}
\begin{proof}
Let $(\phi_i)_{i \in\mathbb{N}}$ be dense subset in $X$. Then, it is sufficient to show 
\[
\int_S\langle u_{n_k},\phi_i\rangle \psi d\mu \to \int_S\langle u,\phi_i\rangle \psi d\mu \quad \quad 
\forall\, i\in\mathbb{N}, \, \forall \,\psi\in L^\frac{p}{p-1}(S,\mathbb{R})
\]
for some subsequence $(n_k)$ and $u \in L^p(S,X^*)$.
Observe that, in view of H\"older's inequality and \eqref{eq:uniform_bound}, we have
\[
\int_S|\langle u_{n},\phi_i\rangle|^p d\mu \leq \int_S|u_{n}|_{X^*}^p|\phi_i|_X^p d\mu <  C|\phi_i|_X^p
\]
for some constant $C$ independent of $n$. Thus, $\langle u_{n},\phi_1 \rangle$ is a uniformly bounded sequence in the reflexive space $L^p(S,\mathbb{R})$. Therefore, there exists a subsequence $(n_1)$ and $\xi_1\in L^p(S,\mathbb{R})$ such that
\[
\int_S\langle u_{n_1},\phi_1\rangle \psi \, d\mu \to \int_S\xi_1 \psi \, d\mu \quad \quad \forall \, \psi\in L^\frac{p}{p-1}(S,\mathbb{R}).
\]
Repeating the above process with each $\phi_i$ and subsequence obtained from previous step, there exists a subsequence $(n_k)$ and $(\xi_i)_{i\in\mathbb{N}}$ such that
\[
\int_S\langle u_{n_k},\phi_i\rangle \psi \, d\mu \to \int_S\xi_i \psi \,d\mu 
\quad \quad \forall \, i\in\mathbb{N},\,\,\forall \, \psi\in L^\frac{p}{p-1}(S,\mathbb{R}).
\]
Finally, one can define $u\in L^p(S,X^*)$ by
\[
\int_S\langle u,\phi_i \psi\rangle \, d\mu := \int_S\xi_i \psi \, d\mu \quad \quad \forall \, i\in\mathbb{N}, \, \forall \, \psi\in L^\frac{p}{p-1}(S,\mathbb{R})
\]
and note that,
\[
\int_S\langle u_{n_k},\phi_i\rangle \psi \, d\mu \to \int_S\xi_i \psi \, d\mu =\int_S\langle u,\phi_i \rangle \psi\, d\mu \quad \quad \forall \, i\in\mathbb{N},\,\,\forall \, \psi \in L^\frac{p}{p-1}(S,\mathbb{R})
\]
as desired.
\end{proof}

\end{document}